\documentclass[10pt]{article}

\usepackage{amssymb,amsmath,amsthm,amscd,enumerate}
\usepackage{verbatim}
\usepackage[francais]{babel}
\usepackage[all]{xy}
\usepackage[T1]{fontenc}

\usepackage[colorlinks,linkcolor=black,citecolor=black,urlcolor=red]{hyperref} 

\date{}
\newcommand{\n}{\noindent}
\newcommand{\rrr}{\longrightarrow}
\newcommand{\s}{{\rm Spec}}
\newcommand{\m}{{\sf M}} 
\newcommand{\p}{{\mathbf P}}

\newcommand{\oo}{{\mathcal O}}

\newcommand{\bb}{\bigskip}
\newcommand{\e}{{\sf E}}
\newcommand{\Et}{{\operatorname{\acute{e}t}}}

\newcommand{\scd}{sch\'ematiquement dominant}

\newcommand{\rt}{\, \widetilde{\longrightarrow}\, }
\newcommand{\wt}{\widetilde}

\newcommand{\sm}{\operatorname{\mathsf{Sm}}}
\newcommand{\ppf}{\operatorname{\mathsf{Pl{.}pf}}}
\newcommand{\et}{\operatorname{\mathsf{Et{.}sep}}}

\newcommand{\Spec}{\operatorname{Spec}}

\newcommand{\IM}{\operatorname{Im}}

\newcommand{\af}{^{{\rm aff}}}
\newcommand{\by}{\xrightarrow}

\newcommand{\iso}{\by{\sim}}

\newcommand{\iv}{^{-1}}
\newcommand{\Hom}{\operatorname{Hom}}
\newcommand{\qcqs}{quasi-compact et quasi-séparé \,}

\newcommand{\of}{\mathsf{Of}}
\newcommand{\cf}{\mathsf{Cf}}
\renewcommand{\epsilon}{\varepsilon}

\renewcommand{\phi}{\varphi}
\newcommand{\sfC}{\mathsf{C}}
\newcommand{\sfD}{\mathsf{D}}

\newcommand{\sfR}{\mathsf{R}}
\newcommand{\PI}{\pi^s}
\newcounter{spec}
\newenvironment{thlist}{\begin{list}{\rm{(\roman{spec})}}%
{\usecounter{spec}\labelwidth=20pt\itemindent=0pt\labelsep=10pt}}%
{\end{list}}%

\numberwithin{equation}{subsection}

\swapnumbers
\newtheorem{lemme}{Lemme}[subsection]
\newtheorem{thm}[lemme]{Théorème}
\newtheorem{prop}[lemme]{Proposition}
\newtheorem{cor}[lemme]{Corollaire}

\theoremstyle{definition}
\newtheorem{defn}[lemme]{Définition}

\theoremstyle{remark}
\newtheorem{rque}[lemme]{Remarque}
\newtheorem{cons}[lemme]{Construction}
\newtheorem{rques}[lemme]{Remarques}
\newtheorem{ex}[lemme]{Exemple}
\newtheorem{exs}[lemme]{Exemples}

\newtheorem{para}[lemme]{}

\title{Un adjoint}

\begin{document}
\author{Daniel Ferrand}
\date{}
\maketitle
\begin{abstract}
Soit $S$ un schéma, disons  irréductible. On montre l'existence d'un foncteur qui associe à tout morphisme plat et de présentation finie  $T\to S$, un $S$-schéma $\PI(T/S)$ étale quasi-compact \emph{et séparé} (en bref, un objet de $\et_{S}$), et un $S$-morphisme $h_{T}: T \to \PI(T/S)$ qui sont \og universels \fg \, pour les $S$-morphismes de $T$ vers un objet de $\et_{S}$. L'étude du foncteur $T \mapsto \PI(T/S)$ utilise les relations d'équivalence dans $T$, à graphe ouvert et fermé dans $T\times_{S}T$.
On montre, en particulier, que lorsque $S$ est normal intègre et que $T$ est lisse sur $S$, alors la formation de $\PI(T/S)$ commute à la restriction aux ouverts de $S$.
Enfin, {\sc Laumon} et {\sc Moret-Bailly}, dans le cadre élargi des $S$-espaces algébriques lisses,  ont introduit l'adjoint à gauche $\pi_{0}(T/S)$, à valeur dans les espaces algébriques étales. On montre que, lorsque $S$ est normal,  le morphisme d'espaces algébriques   $\pi_{0}(T/S) \to \pi^s(T/S)$ fait de  $\pi^s $ l'enveloppe séparée de  $\pi_{0}$. 
\begin{center}
{\bf Abstract}
\end{center}
Let $ S $ be a scheme, say irreducible. We prove the existence of a functor which associates with any flat morphism of finite presentation $ T \to S $, a $S$ -scheme $ \PI (T / S) $ étale quasi-compact  \emph {and separated } (in short, an object of $ \et_ {S} $) and a $S$-morphism $ h_ {T}: T \to \PI (T / S) $ which are \og universal \fg \, for the $S$-morphisms from $T$ to an object of $\et_ {S}$. The study of the functor  $T \mapsto \PI (T/S) $ rests equivalence relations in $T$, whose graph is open and closed  in $ T \times_ {S} T $. We show, in particular, that when $S$ is normal and $T$ is smooth over $S$, then the formation of $ \PI(T/S) $ commutes with the restriction  to  open sets in $S$. Finally, {\sc Laumon} and {\sc Moret-Bailly}, in the extended context of $S$-smooth algebraic spaces, introduced the left adjoint  $\pi_ {0}(T/S)$, with value  étale algebraic spaces. When $ S $ is normal, we show that the morphism of algebraic spaces $\pi_ {0}(T/S) \to \pi^s(T/S)$ makes $ \pi^s $ the separated envelope of $ \pi_ {0} $.
\end{abstract}

\tableofcontents

\section*{Introduction}

Soit $S$ un schéma  dont l'ensemble des composantes irréductibles est fini (c'est le cas si $S$ est intègre ou noethérien); tous les morphismes $T\to S$ évoqués dans cette introduction seront implicitement supposés  plats et de présentation finie; leur catégorie sera désignée par $\ppf_S$.

 Nous montrons d'abord que le foncteur d'inclusion de la catégorie des $S$-schémas étales et séparés, dans $\ppf_S$ admet un adjoint à gauche; autrement dit, pour tout  objet $f : T \rrr S$  de $\ppf_S$,  il existe un morphisme de $S$-schémas $h_{T}: T \to \pi^s(T)$  qui est universel pour les $S$-morphismes de $T$ vers des schémas étales  et séparés sur $S$.\\ 

 L'existence de cet adjoint était connue dans au moins deux contextes:
1) Lorsque $S$ est le spectre d'un corps $k$ et que $T\to S$ est de type fini:  l'adjoint est  alors le spectre de la fermeture algébrique séparable de $k$ dans $\Gamma(T, \oo_{T})$.

2) Pour un morphisme propre et lisse la factorisation de Stein représente l'adjoint:  $\s_{S}(f_{\star}(\oo_{T}))$ est étale fini sur $S$ et le morphisme $h: T \to \s_{S}(f_{\star}(\oo_{T}))$ est universel pour les morphismes de $T$ vers des étales, séparés ou non, sur $S$.\\

Enfin, il y a quelques temps,  {\sc Bruno Kahn}  construisit  l'adjoint à gauche pour un morphisme lisse sur une base de Dedekind.\\

Le morphisme universel $h_{T}: T \to \pi^s(T)$ est, par définition, l'objet initial de la catégorie des {\it factorisations} de $f: T \to ~S$, ce terme étant, dans ce texte, réservé aux  suites de morphismes $T\by h E \by g S$, tels que $f = gh$, où $g$  est étale et séparé, et où $h$ est surjectif. Sous les hypothèses faites au début sur $S$ et sur $T\to S$, cette catégorie des factorisations, notée $\e(T/S)$, est équivalente à un ensemble ordonné {\it fini} et filtrant à gauche; cela rend évidente l'existence de l'adjoint, mais ne donne guère de prise sur ses propriétés. 

Il s'avère que la catégorie des factorisations $\e(T/S)$ est plus maniable que son objet initial; elle est définie sans hypothèse sur $S$, mais, surtout,   sa fonctorialité est évidente: un morphisme  $u: T' \to T$   dans $\ppf_S$ induit, en effet, un foncteur $\e(u): \e(T/S) \to \e(T'/S)$, dont voici la définition;
pour une factorisation $T\by h E \by g S$ de $f$, le morphisme $hu$ est ouvert et s'écrit donc $u'h'$, avec $h'$ surjectif et $u'$ une immersion ouverte
$$
\xymatrix{T' \ar[r]^u \ar[d]_{h'} & T\, \ar[d]^h\\
E' \ar[r]_{u'} & E\, .}
$$
L'image par $\e(u)$ de la factorisation $(g, h)$ est la factorisation $(gu', h')$. De plus, on montre que le foncteur $\e(u)$ est une équivalence si et seulement si le morphisme $\pi^s(u): \pi^s(T') \to \pi^s(T)$ est un isomorphisme.\\  

Pour déterminer si un morphisme $u$ induit  une équivalence $\e(u)$, on utilise le point de vue des relations d'équivalence. En effet, pour une factorisation  $T\by h E \by g S$ de $f: T\to S$, au sens restreint donné plus haut, le morphisme $h$ est fidèlement plat de présentation finie, donc la suite 
$$
T\times_{E} T\rightrightarrows T\by{h} E
$$
est exacte dans la catégorie des schémas; le morphisme $h: T \to E$ apparaît ainsi comme le quotient  de $T$ par la relation d'équivalence $R=T\times_{E}T \rightrightarrows T$; ce schéma $R$ est ouvert et fermé dans $T\times_{S}T$. Réciproquement, pour une relation d'équivalence $R$, à graphe ouvert et fermé, on démontre (de  trois fa\c cons !) que le faisceau quotient fppf $T/R$ est représentable par un schéma étale, quasi-compact et séparé sur $S$. 

Cela conduit à considérer l'ensemble, ordonné par inclusion, $\of(T\times_{S}T)$, des sous-schémas ouverts et fermés du produit $T\times_{S}T$; les relations d'équivalence envisagées en forment un sous-ensemble ordonné. Un $S$-morphisme $u: T' \to T$ induit par image inverse une application
$$
(u\times u)^* : \of(T\times_{S}T) \to \of(T'\times_{S}T').
$$
On montre l'utile critère suivant: si $u$ est schématiquement dominant et que l'application ci-dessus $(u\times u)^*$ soit bijective, alors  le foncteur $\e(u)$ définit une {\it équivalence} $\e(T/S) \; \by{\sim}\;  \e(T'/S)$ entre les catégories de factorisation, et  le morphisme $\pi^s(u): \pi^s(T'/S) \to \pi^s(T/S)$ est un isomorphisme. \\

Ainsi, un morphisme schématiquement dominant $u: T'\to T$ dans $\ppf_S$ induit un isomorphisme  $\pi^s(u): \pi^s(T'/S) \by{\sim} \pi^s(T/S)$, par exemple, dans les cas suivants:
\begin{itemize}
\item $u$ est universellement submersif (par exemple fpqc) et ses fibres sont géométriquement connexes; en particulier, pour $T' = {\bf A}^n_{T}$, ou ${\bf P}^n_{T}$, etc.
\item $u$ est un homéomorphisme universel;
\item l'application $\oo_{T} \to u_{\star}(\oo_{T'})$ est bijective.
\end{itemize}

De plus, lorsque $S$ est normal intègre et que $T$ est lisse sur $S$, le critère ci-dessus conduit à des propriétés de prolongement: 

Si $u: T'\to T$ est une immersion ouverte dense,  alors $\e(u)$ est une équivalence. 

Et aussi: pour tout ouvert non vide $U \subset S$, le morphisme canonique 
$$
\pi^s(U\times_{S}T / U) \to U \times_{S}\pi^s(T/S)
$$
est un isomorphisme; en passant à la limite sur les ouverts $U$, on montre que, en notant $\xi$ le point générique de $S$,   le morphisme canonique
$$
 \pi^s(T_{\xi}/ \xi) \; \rt \; \pi^s(T/S)_{\xi}
$$
est un isomorphisme. 

Cela permet d'étendre à une base normale $S$, les propriétés  du foncteur adjoint qui sont vraies sur un corps de base; en particulier, le foncteur $\pi^s$ commute  aux produits de $S$-schémas lisses.\\

On peut en dire un peu plus sur les changements de base. Soit $\phi : \wt{S} \to S$ un morphisme entre des schémas dont les composantes irréductibles sont en nombre fini, et soit $T\to S$ un morphisme plat et de présentation finie. D'après la propriété universelle de $\pi^s(\wt{S}\times_{S}T/\wt{S})$, on a  un morphisme de $\wt{S}$-schémas
$$
\pi^s(\wt{S}\times_{S}T/\wt{S}) \to \wt{S} \times_{S}\pi^s(T/S).
$$
Ce n'est en général pas un isomorphisme, même si $\phi$ est une immersion ouverte. Cependant, on montre que c'est un isomorphisme si l'application induite par $\phi^*$, $\of(S) \to \of(\wt{S})$ est universellement bijective. C'est le cas si $\phi$ est universellement submersif et si ses fibres sont géométriquement connexes, ou bien si $\phi$ est quasi-compact et quasi-séparé, et que l'application $\oo_{S} \to \phi_{\star}(\oo_{\wt{S}})$ est bijective.

C'est aussi un isomorphisme si $\phi : \wt{S} \to S$ est un morphisme dominant de schémas normaux intègres et que $T$ soit lisse sur $S$, car on peut, ici encore, se ramener à l'extension de corps générique.\\

Enfin, toujours en supposant que $S$ est normal intègre, on montre que pour un morphisme étale de présentation finie $T\to S$, le morphisme universel $h_{T}: T \to \pi^s(T/S)$ est un isomorphisme local (local sur $T$), et qu'il fait de $\pi^s(T/S)$ l'enveloppe séparée de $T/S$. 

Considérons alors l'espace algébrique étale $\pi_{0}(T/S)$ qui représente les composantes connexes des fibres géométriques de $T\to S$ (\cite[6.8]{LMB00}); le morphisme $T \to \pi_{0}(T/S)$ est universel pour les $S$-morphismes  $T\to E$, avec $E$ étale, séparé ou non, sur $S$; il existe donc un morphisme canonique d'espaces algébriques $\theta: \pi_{0}(T/S)  \to \pi^s(T/S)$. On montre que si $T \to S$ est lisse de présentation finie, alors $\theta$ fait de $\pi^s(T/S)$ l'enveloppe séparée de $\pi_{0}(T/S)$ (dans la catégorie des $S$-espaces algébriques); il est  vraisemblable que $\theta$ soit aussi un isomorphisme local comme dans le cas schématique évoqué plus haut.\\

Cet article doit beaucoup à {\sc Bruno Kahn} : comme signalé plus haut, il avait construit cet adjoint lorsque la base est de Dedekind et que le morphisme est lisse; et il m'avait demandé si on pouvait étendre sa construction sur une base générale. Il est apparu que le procédé qu'il utilisait est trop dépendant de la dimension 1 pour pouvoir être généralisé, et qu'il fallait donc trouver autre chose.


\section{Préliminaires}

\subsection{Définitions}\label{l0.1}
\begin{par}
Les définitions et les notations adoptées sont celles des E.G.A., et de la nouvelle édition pour EGA I.\
 Voici le rappel de ce qui sera le plus souvent utilisé.
 
\end{par}

\begin{para} \label{p.finie}Conformément aux conventions de EGA I, un morphisme $f: Y \to X$ est \emph{de présentation finie} s'il est localement de présentation finie, quasi-compact et quasi-séparé; ce dernier terme signifie que le morphisme diagonal $\Delta_{f}: Y \to Y\times_{X}Y$ est quasi-compact \cite[6.1]{EGAI}.\

La propriété pour un morphisme d'être {\it quasi-scompact et quasi-séparé} assure la quasi-cohérence des images directes, et elle est nécessaire dès que des adhérences schématiques interviennent, ce qui est fréquent dans la suite (\ref{im.sch}).
\end{para}

\begin{para}
Parmi les définitions possibles de \emph{morphisme étale}, nous utiliserons surtout celle-ci : un morphisme de schémas $X \to S$ est {\it étale} (resp. {\it étale et séparé}) s'il est plat, localement de présentation finie et si le morphisme diagonal $X \to X\times_{S}X$ est une immersion ouverte (resp. une immersion ouverte et fermée). 
 L'équivalence entre cette définition et les autres est exposée dans  \cite[17.4.2 et 17.6.2]{EGAIV4}.
 
 Un morphisme étale et séparé est de présentation finie si et seulement si il est quasi-compact.
 \end{para}

\begin{para}
({\it Exemple de morphisme étale non séparé: dédoubler un point.}) \phantomsection\label{pr1.3}  Soient $f: T \to S$ un morphisme, et $s \in S$ un point fermé tel que la fibre $T_{s} = f^{-1}(s)$ ne soit pas connexe; elle est donc  la réunion disjointe d'au moins deux sous-schémas non vides:   $T_{s} = D_{1} \sqcup D_{2}$; ils sont fermés dans $T_{s}$, donc fermés aussi dans $T$. Les ouverts complémentaires $V_{i} = T - D_{i}$ ont les propriétés suivantes:
 \begin{itemize}
 \item $V_{1}\cup V_{2} = T, \; V_{1}\cap V_{2} = T - T_{s}$.
 \item Les morphismes $V_{i}\to S$ induits par $f$ sont surjectifs, et on a $f(V_{1}\cap V_{2}) = S - s$.
 \end{itemize}
 
 Introduisons le schéma $F$ obtenu, à partir de deux copies de $S$, par leur recollement le long de $S-s$ (parfois nommé: \emph{le schéma $S$ avec le point $s$ dédoublé}).\
 
 Le morphisme $f$ se factorise en $T \to F \to S$, et le morphisme $F\to S$ est étale; il est séparé si et seulement si $s$ est ouvert dans $S$.
  \end{para}
  
.\begin{lemme}[{\rm Changement de base et image directe} \protect{\cite[9.3.3]{EGAI}}]\label{pr1.1} Soient $f: Y \to X$ et $g: X'\to X$ des morphismes; on note $f':Y'=X'\times_{X}Y \to X'$ le morphisme obtenu par changement de base. 
$$
\xymatrix{Y  \ar[d]_{f} & Y'=Y\times_{X}X'  \ar[l]_-{g'}\ar[d]^{f'}\\
X& X' \ar[l]^{g}}
$$
On suppose que $f$ est quasi-compact et quasi-séparé, et que $g$ est plat. Alors, on dispose d'un isomorphisme canonique
$$
w: g^\star(f_{\star}(\oo_{Y})) \to f'_{\star}(g'^\star(\oo_{Y})) = f'_{\star}(\oo_{Y'}) .
$$
\label{s.schdom} En particulier, si $f$ est \scd, {\it i.e.} si l'homomorphisme $\oo_{X} \to f_{\star}(\oo_{Y})$ est injectif \emph{\cite[5.4.1]{EGAI}}, alors $f'$ est \scd.
\end{lemme}

 \begin{para} \label{im.sch} Soit $f : Y \to X$ un morphisme quasi-compact et quasi-séparé, de sorte que $f_{\star}(\oo_{Y})$ est une $\oo_{X}$-algèbre quasi-cohérente. {\it L'image schématique} de $f$ est le sous-schéma fermé  $j: X' \to X$ de $X$ défini par l'idéal 
 $$
 \mathcal{J} = {\rm Ker}(\oo_{X} \to f_{\star}(\oo_{Y}))
 $$
 L'espace sous-jacent à $X'$ est égal à l'adhérence $\overline{f(Y)}$, et $f$ se factorise en
 $$
 Y \by g X' \by j X
 $$
 où $g$ est schématiquement dominant. \cite[6.10.5]{EGAI}. Cette construction commute à tout changement de base plat (\ref{s.schdom}.)\
\end{para}
\begin{para} 
 \label{adh.sch} Lorsque le morphisme $f : Y \to X$ est un sous-schéma quasi-compact, on parle {\it d'adhérence schématique} \cite[6.10.6]{EGAI}. 
 Dans ce cas, le morphisme $Y \to X'$ du sous-schéma $Y$ dans son adhérence schématique est une immersion ouverte \scd e \cite[5.4.4]{EGAI}.
 Ici encore, cette construction commute à tout changement de base plat (\ref{s.schdom}).
\end{para}
  
 \begin{para}\phantomsection\label{pr1.2}
Un morphisme plat et localement de présentation finie est universellement ouvert  \cite[7.3.10]{EGAI} ou bien \cite[2.4.6]{EGAIV2}.
 \end{para}

 \subsection{ Adjonction et morphisme universel}\label{univ}
 
 Soit {\sf C} une sous-catégorie pleine d'une catégorie {\sf D}. Un adjoint à gauche de cette inclusion de catégories est un foncteur $F: {\sf D} \to {\sf C}$  muni d'un isomorphisme de bifoncteurs, pour $X$ dans {\sf C}  et $Y$  dans {\sf D},
 $$
 {\rm Hom}_{\sf{C}}(F(Y), X)\iso {\rm Hom}_{\sf{D}}(Y, X)  .
 $$
 Autrement dit, ce foncteur associe, à tout objet $Y$ 	dans {\sf D} un objet $F(Y)$ dans {\sf C} et un morphisme 
 $$
 h_{Y} : Y \to F(Y)
 $$
  vérifiant la propriété suivante: pour tout $X\in \sfC$ et tout morphisme $v: Y \to X$ dans $\sfD$,  il existe un unique morphisme $u: F(Y)\to X$ tel que $v = u\small{\circ} h_{Y}$. Par abus de langage, on écrira souvent que \emph{$h$ est universel pour les objets de $\sfC$.}  On emploiera aussi l'expression: le morphisme $h_{Y} : Y \to F(Y)$ fait de $F(Y)$ {\it l'enveloppe de} $Y$ dans {\sf C.}
  
 Rappelons trois propriétés importantes de l'adjoint:
  
 \begin{itemize}
 \item Il est défini point par point: $F(Y)$ coreprésente le foncteur $\sfC\ni X\mapsto \Hom_{\sfD}(Y,X)$. Ainsi, l'expression \og $F$ est défini en $Y$ \fg\,  a un sens.
 \item Pour tout $Y\in \sfC$, $F$ est défini en $Y$ et $h_Y$ est un isomorphisme: cela résulte de la pleine fidélité de l'inclusion $i:\sfC\to \sfD$.
 \item Il commute aux limites inductives (représentables) quelconques, en particulier aux coproduits.
 \end{itemize}

 \subsection{Trois lemmes}
\begin{lemme}\phantomsection\label{l1.2} Soit $\sf C$ une catégorie admettant des produits fibrés (par exemple la catégorie des schémas), et soit $\p$ une propriété des morphismes de $\sf C$, stable par composition et par changement de base. Soit
\[X\by{\alpha}Y\by{\beta} Z \]
une suite de morphismes. On suppose que le composé $\beta \alpha$ vérifie $\p$, ainsi que le morphisme diagonal $\Delta_{\beta} : Y\to Y\times_Z Y$. Alors $\alpha$ vérifie $\p$.
\end{lemme}

\begin{proof} Factorisons $\alpha$ en
\[X\by{\gamma} X\times_Z Y \by{p} Y\] 
où $\gamma$ est le graphe de $\alpha$ et où $p$ est la seconde projection. Alors $\gamma$ se déduit de $\Delta_{\beta}$ par le changement de base  $X\times_Z Y\by{\alpha\times_Z 1_Y} Y\times_Z Y$, et $p$ se déduit de $\beta \alpha$ par le changement de base par $\beta$ . 
\[\xymatrix{
&X\ar[r]_{\beta\alpha}&Z\\
X\ar[r]^-{\gamma}\ar[d]^\alpha& X\times_Z Y\ar[d]^{\alpha\times_Z 1_Y}\ar[r]^-p\ar[u]& Y\ar[u]_\beta\\
Y\ar[r]^-{\Delta_\beta}& Y\times_Z Y
}\]

D'où la conclusion.  \end{proof}

\begin{lemme}[Prolongement d'isomorphismes]\phantomsection\label{l1.1}  Soit $u: X \to Y$ un morphisme fidèlement plat, quasi-compact et \emph{séparé}. Soit $i : Y' \to Y$ un morphisme quasi-compact, quasi-séparé et \scd; notons $u' : X' = Y'\times_{Y}X\to Y'$ le morphisme déduit de $u$ par changement de base; on suppose que $u'$ est un isomorphisme.
Alors $u$ est un isomorphisme.
\end{lemme}

L'hypothèse sur le morphisme $i: Y'\to Y$ est vérifiée en particulier si c'est une immersion ouverte quasi-compacte qui est, de plus,  \scd e,  ou bien si $Y$ est 
réduit et que ses points maximaux $\xi_{\lambda}$ soient en nombre fini, et, enfin,  si  $i$ est le morphisme $\bigsqcup \s(\kappa(\xi_{\lambda})) \to Y$.

{\color{black} Cet énoncé étend légèrement le cas affine bien connu, qui s'énonce ainsi : soit $A \to B$ un homomorphisme fidèlement plat d'anneaux intègres. Si $B$ est contenu dans le corps des fractions de celui de $A$, alors $A = B$. (voir  \cite[I, \S3.5, Prop. 9, $b)$]{AC})}

\begin{proof}  
Par ``descente fpqc '' le long de $u$  \cite[2.7.1]{EGAIV2}, il suffit de montrer que l'une des projections $X\times_{Y}X \to X$ est un isomorphisme ; or, ces morphismes admettent une section, à savoir le morphisme diagonal $\Delta_{u} : X \to X\times_{Y}X$ ; il s'agit donc de montrer que $\Delta_{u}$ est un isomorphisme. 
 Puisque $u$ est supposé séparé, son morphisme diagonal $\Delta_{u}$ est  une immersion fermée; il suffit  donc de vérifier  que $\Delta_{u}$ est \scd. Comme le morphisme $i : Y' \to Y$ est quasi-compact, quasi-séparé et \scd\,  et que $u$ est plat, les morphismes verticaux du diagramme suivant sont, eux aussi, \scd s (\ref{s.schdom}).
$$
\xymatrix{ X \ar[r]^-{\Delta_{u}} & X\times_{Y}X \\
X' \ar[u] \ar[r]_-{\Delta_{u'}} & X'\times_{Y'}X' \ar[u]}
$$
Par hypothèse, $\Delta_{u'}$ est un isomorphisme, la commutativité du diagramme implique donc  que $\Delta_{u}$ est \scd. 
\end{proof}


\begin{lemme}\phantomsection\label{p1.1} Soient $f : T \rrr S$ un morphisme plat de pr\'esentation finie et $T \stackrel{h}{\rrr} E \stackrel{g}{\rrr} S$ une factorisation de $f$, o\`u $g$ est \'etale de pr\'esentation finie (éventuellement non séparé), et où $h$ est surjectif. Alors les conditions suivantes sont équivalentes. 
\begin{thlist}
\item Les fibres du morphisme $h$ sont g\'eom\'etriquement con\-nexes.
\item Le morphisme $h$ est universel (\ref{univ}) pour les $S$-morphismes de $T$ vers un $S$-sch\'ema \'etale  \emph{(non nécessairement séparé)} et de pr\'esentation finie, et il reste universel apr\`es tout changement de base $S' \rrr S$.
\item Pour tout point géométrique $\omega = \s(\Omega) \to S$ de $S$, le $\omega$-morphisme $h_{\omega}: T\times_{S}\omega \to E\times_{S}\omega$ est universel pour les $\omega$-morphismes de $T\times_{S}\omega$ vers un $\omega$-sch\'ema \'etale  de type fini. Dans les termes de l'énoncé suivant (\ref{p1}), cela s'écrit $\pi_{0}(T\times_{S}\omega / \omega) \simeq E\times_{S}\omega$.
\end{thlist}
\end{lemme}

\begin{proof}
 ${\rm (i)} \Rightarrow {\rm (ii)}$ Les hypoth\`eses sur $h$ \'etant stables par changement de base, il suffit  de montrer que $h$ est universel au-dessus de $S$. Soit $f = g'h'$ une factorisation avec $g' : E' \rrr S$ étale de présentation finie.  Il s'agit de montrer qu'il existe un morphisme  $v : E \to E'$ tel que $h' = vh$  et $g = g'v$.  Isolons, pour cela, une construction qui resservira.
 
 \begin{cons}\phantomsection\label{c.1} Soit
 $$
 \xymatrix{T \ar[r]^{h} \ar[d]_{h'} & E \ar[d]^{g} \ar@{-->}[dl]\\
 E'\ar[r]_{g'} &S}
 $$
 un carré commutatif de morphismes de schémas, où $h$ est universellement ouvert et où $g'$ est étale. On cherche à construire un morphisme $E\to E'$ rendant les deux triangles commutatifs.
 \end{cons}
 Notons $h'' : T \to E'\times_{S}E$ le morphisme déduit de $h'$ et de $h$, de sorte que le composé $T\by{h''} E'\times_{S} E\by{g'\times_S 1_E} E$ est égal à $h$.  Comme $h$ est universellement ouvert, ainsi  que le morphisme diagonal $\Delta_{g'}$,  le lemme \ref{l1.2} montre que $h''$ est lui aussi universellement ouvert; son image $U=h''(T) \subset E'\times_{S}E$ est donc un ouvert; on garde la lettre  $h''$ pour désigner le morphisme surjectif $T \to U$ déduit de $h''$ ; ce schéma $U$ s'insère dans le diagramme commutatif ci-dessous,  où $u'$ est le morphisme composé $U \subset E'\times_{S}E \stackrel{{\rm pr}_{1}}{\rrr} E'$ .
 $$
 \xymatrix{T\ar[drr]^h \ar[dr]|{h''} \ar[ddr]_{h'} &&\\
 & U \ar[r]_{u} \ar[d]^{u'} & E \ar[d]^g\\
 &E' \ar[r]_{g'} & S .} 
 $$
 Si on a  pu vérifier que $u$ est un isomorphisme, alors le morphisme cherché sera le composé $u'\small{\circ} u^{-1} : E \to U \to E'$.
 \\
 
 Revenons à la démonstration du lemme.\
 
  Comme le morphisme $g'$ est supposé étale, le morphisme $u$ est étale, et il est surjectif puisque $h = uh''$ l'est; il suffit donc, pour pouvoir conclure, de montrer que pour tout point $x$\footnote{Ici, et dans la suite, la lettre \og $x$\fg désignera souvent le schéma $\s(\kappa(x))$; le symbole $u^{-1}(x)$ désigne alors la fibre schématique
  $$
 u^{-1}(x) := U\times_{E}\s(\kappa(x)).$$} de $E$, le morphisme $u^{-1}(x) \to x$  est un isomorphisme \cite[17.9.1]{EGAIV4}. Or le morphisme composé 
$$
h^{-1}(x)= h''^{-1}u^{-1}(x)  \rrr u^{-1}(x) \rrr x
$$
 est géométriquement connexe, et $h''$ est surjectif ; donc le morphisme $u^{-1}(x)\allowbreak \to x$ est géométriquement connexe et étale ; c'est un isomorphisme.

(ii) $\Rightarrow$ (iii) est trivial. 
 
 (iii)$\Rightarrow$ (i) : 
On suppose maintenant que dans la factorisation  $f=gh: T \to E \to S$ le morphisme $h$ est universel pour les $S$ schémas étales, et le reste par changement de base aux points géométriques de $S$. Il s'agit de montrer que les fibres de $h$ sont géométriquement connexes, c'est-à-dire (\cite[4.5.2]{EGAIV2}), que pour tout point géométrique $\omega = \s(\Omega) \to E$ ($\Omega$ un corps algébriquement clos), la fibre $T_{\omega}$ est connexe.
Considérons le changement de base par le morphisme composé  $\omega \to E \to S$, noté $\epsilon$.
$$
\xymatrix{T \ar[r]^h & E \ar[r]^g &S\\
T\times_{S}\omega \ar[u] \ar[r]_{h_{\epsilon}} &E\times_{S}\omega \ar[u] \ar[r]_{g_{\epsilon}} & \omega \ar[u]_{\epsilon}}
$$
Le $\omega$-schéma étale $E\times_{S}\omega$ est une somme  finie $\bigsqcup_{i}\omega_{i}$ de copies de $\omega$, l'une d'elles, $\omega_{0}$, étant donnée par le morphisme diagonal $\omega \to E\times_{S}\omega$; le schéma $T\times_{S}\omega$ se décompose donc  en la somme des fibres $ T_{i} =  h_{\epsilon}^{-1}(\omega_{i})$ du morphisme $h_{\epsilon}$, lequel s'écrit  $h_{\epsilon} = \sqcup h_{i}$; il faut voir que la fibre $T_{0}$ est connexe. En fait, c'est vrai pour chaque $T_{i}$; en effet,  la décomposition $h_{\epsilon} = \sqcup h_{i}$  montre que chaque morphisme $h_{i}: T_{i} \to \omega_{i} $ hérite de la propriété universelle de $h_{\epsilon}$, propriété  qui signifie ici que la source est connexe puisque le corps $\omega_{i}$ est algébriquement clos.
\end{proof}

\section{Cas connus}

\subsection{Sur un corps}

\begin{prop}\phantomsection\label{p1} Soient $S = \s(k)$ le spectre d'un corps et $f : T \rrr S$ un morphisme {\color{black} localement}  de type fini. Alors,
\begin{thlist}
\item  Il existe un $S$-sch\'ema \'etale, not\'e $\pi_{0}(T/S)$, et une factorisation de $f$ en
$$
T \; \stackrel{h}{\rrr} \; \pi_{0}(T/S) \; \rrr \; S ,
$$
où le morphisme $h$ est universel pour les $S$-morphismes de $T$ vers un $S$-sch\'ema \'etale.

Lorsque $T = \s(A)$ est affine et de type fini,  $\pi_{0}(T/S)$ est le spectre de la clôture algébrique séparable de $k$ dans $A$; c'est une $k$-algèbre étale finie. 

\item Si $S' \to S$ d\'esigne le morphisme de sch\'emas associ\'e \`a une extension du corps de base, le morphisme canonique
$$
  \pi_{0}(S'\times_{S}T / S')\; \rrr \; S'\times_{S}\pi_{0}(T/S).
$$
est un isomorphisme.

\item Les fibres de $h$ sont g\'eom\'etriquement connexes.

\item Si $T'\to S$ est un second $S$-sch\'ema localement de type fini, le morphisme canonique
$$
\pi_{0}(T\times_{S}T'/S) \;  \rrr \;  \pi_{0}(T/S)\times_{S}\pi_{0}(T'/S).
$$
est un isomorphisme.
 \end{thlist}
\end{prop}

Toutes les propriétés énoncées sont démontrées dans le traité de  {\sc Demazure} et {\sc Gabriel}; voir  \cite[I, \S4, n$^o$ 6, p.122-126]{DG70}.

Notons qu'il a été démontré, pour le lemme précédent (\ref{p1.1}), que la propriété (i) (universalité de $h$) implique (iii).  

L'énoncé (\ref{p10}) donne une autre description de $\pi_{0}(T/S)$.

Un schéma étale sur le spectre d'un corps est discret donc séparé; sur une base générale la séparation des $S$-schémas étales considérés sera requise. \\
 
\begin{rque} Soient $k$ un corps et $k \to A$ une algèbre de type fini intègre et normale, de corps des fractions $K$. Les éléments de $K$ qui sont algébriques et séparables sur $k$, sont en particuliers entiers sur $k$; ils sont donc dans $A$ puisque $A$ est intégralement fermé dans $K$; passant aux spectres, et en notant $\eta = \s(K)$, cela s'écrit, pour tout ouvert non vide $U$ de $T=\s(A)$,  $\pi_{0}(\eta/S) \by{\sim}\pi_{0}(U/S) \by{\sim} \pi_{0}(T/S)$. Cela sera généralisé en \ref{p11}.
 
 La normalité est essentielle pour les résultats de ce genre; il est bon de garder présent à l'esprit le schéma $T$ formé de deux droites sécantes et  l'ouvert $U \subset T$ complémentaire du point d'intersection, c'est-à-dire $U = \s(k[X]_{X}\times k[Y]_{Y}) \subset \s(k[X, Y]/(XY)) = T$ ; le morphisme $\pi_{0}(U/S) \to \pi_{0}(T/S)$ n'est pas injectif (oublions donc ce qui est écrit p.16, ligne 6, de \cite{Rom11}).
 \end{rque}
 

\subsection{Sur une base normale de dimension $1$}

\begin{prop}[{\sc Bruno Kahn}]\phantomsection\label{p2}
 Soit $S$ un schéma de Dedekind (noethérien, normal de dimension de Krull $1$).Un morphisme lisse de présentation finie  $f : T \rrr S$ se fatorise en
$$
T \stackrel{h}{\rrr} E  \stackrel{g}{\rrr} S ,
$$
o\`u $g$ est étale \emph{séparé}, et o\`u $h$ est surjectif et universel pour les $S$-morphismes vers un schéma étale et \emph{séparé} sur $S$.
\end{prop}
 
  La démonstration utilise la normalité de $T$, conséquence de la lissité de $f$; mais on montrera en  (\ref{T1}), que la conclusion reste vraie sous des hypothèses beaucoup plus générales.

\begin{proof} On peut supposer que $S$ est connexe, donc intègre ; soit $\eta$ son point  générique.  
Soit $T_{\eta} \rrr F_{0} \rrr \eta $ la factorisation fournie par la proposition \ref{p1}, de sorte que $F_{0} = \pi_{0}(T_{\eta} / \eta)$ est un schéma somme d'une famille finie de spectres de corps extensions finies séparables de $\kappa(\eta)$. Soit $F$ la fermeture intégrale de $S$ dans $F_{0}$; c'est un schéma de Dedekind fini sur $S$. Puisque $S$ est intégralement fermé dans $\eta$, la lissité de $f$ assure que $T$ est intégralement fermé dans $T_{\eta}$; le morphisme générique $T_{\eta} \rrr F_{0}$ se prolonge donc en un morphisme $ T \rrr F$ qui est plat puisque $F$ est de Dedekind ; ce morphisme est aussi de type fini puisque $T$ est de type fini sur $S$ \cite[ch. I, 6.3.6]{EGAI} ; son image est donc un ouvert $E \subset F$ \cite[2.4.6]{EGAIV2}, et $f$ se factorise en
$$
T \stackrel{h}{\rrr} E  \stackrel{g}{\rrr} S
$$
o\`u $h$ est fidèlement plat et $g$ séparé 
et quasi-fini sur $S$ (comme ouvert d'un fini 
plat).  D'après \cite[prop. 17.7.7]{EGAIV4}, $g$ est lisse, donc étale.

 Vérifions  la propriété universelle. Considérons une factorisation $f = g'h'$ où le morphisme $g'$ est supposé étale \emph{et séparé} et reprenons la construction \ref{c.1}, qui aboutit au diagramme 
 $$
 \xymatrix{T\ar[drr]^h \ar[dr]|{h''} \ar[ddr]_{h'} &&\\
 & U \ar[r]_{u} \ar[d]^{u'} & E \ar[d]^g\\
 &E' \ar[r]_{g'} & S .} 
 $$
 Il faut vérifier que $u$ est un isomorphisme.
 Le morphisme  $u$ est  fidèlement plat quasi-compact et séparé ; pour voir que c'est un isomorphisme, il suffit, d'après le lemme \ref{l1.1}, de montrer que le morphisme générique $u_{\eta} : U_{\eta} \to E_{\eta}$ est un isomorphisme ; or, par construction de $E$, le schéma $E_{\eta}= F_{0}$ est isomorphe au schéma $\pi_{0}(T_{\eta} / \eta)$ de la proposition \ref{p1}; la propriété universelle de ce dernier entraîne que $u_{\eta}$ est un isomorphisme.\end{proof}

\begin{rques} 1) Le début de la démonstration montre un peu plus que l'énoncé, à savoir: toute factorisation \og générique\fg\,  $T_{\eta} \rrr F_{0} \rrr \eta $, avec $T_{\eta} \rrr F_{0}$ surjectif et $F_{0} \rrr \eta $ étale, se prolonge en une factorisation de $T\to S$. Nous reviendrons sur cette propriété en \ref{p11}.

2) Sous les hypothèses de \ref{p2}, on ne peut pas espérer que $h$ soit universel pour les morphismes de $T$ vers un étale {\it non séparé} sur $S$. L'exemple de la conique qui se spécialise en deux droites disjointes le montre.
Plus précisemment, soit $S = \s(R)$ le spectre d'un anneau de valuation discr\`ete, d'uniformisante $t$, et tel que $2 \in R^{\times}$. 
 Posons $F(X, Y) = X(X-1)+tY^2$. Le morphisme 
 $$
 f : T = \s(R[X, Y]/(F)) \to S
 $$
est lisse; la fibre spéciale ($t=0$) est la réunion de deux droites disjointes, et la fibre générique $T_{\eta}\to \eta$ est géométriquement intègre (c'est une courbe de genre 0). 
Considérons la factorisation de l'énoncé \ref{p2},  $f = gh$ avec $h$ surjectif et $g$ étale séparé; puisque le morphisme composé $T_{\eta} \by{h_{\eta}} E_{\eta} \by{g_{\eta}} \eta$ est géométriquement intègre et que $h_{\eta}$ est surjectif, le morphisme étale $g_{\eta}$ est un isomorphisme, donc le morphisme $g$ lui-même est un isomorphisme (\ref{l1.1}); autrement dit on a ici $E = S$. 
Mais il existe des factorisations $f = g'h'$ o\`u  $g'$ est étale non séparé puisque la fibre spéciale de $f$ n'est pas connexe \ref{pr1.3}.
\end{rques}


\section{Existence de l'adjoint}\label{s4}

\subsection{La catégorie des factorisations}\label{s4.2}\

\begin{para} \label{not} Fixons les notations.

\begin{description}
\item[$\ppf_S$] la catégorie dont les objets sont les $S$-schémas plats de présentation finie, {\color{black} et dont les morphismes sont tous les $S$-morphismes; par définition, les objets  sont aussi quasi-compacts et quasi-séparés sur $S$ \cite[6.3.7]{EGAI}, et d'après \cite[6.3.8, (v)]{EGAI}, tout $S$-morphisme entre objets de $\ppf_S$ est automatiquement de présentation finie}.
\item[$\et_S$] la sous-catégorie pleine de la précédente dont les objets sont étales, séparés et de présentation finie, c'est-à-dire étales quasi-compacts et séparés sur $S$.
\end{description}

On note $$\iota_{S} : \et_S \to \ppf_S$$ cette inclusion de catégories.

Par définition, un adjoint à gauche de $\iota_{S}$ est un foncteur
$$
\pi^s : \ppf_S \; \rrr \; \et_S
$$
muni d'un isomorphisme de bifoncteurs, pour $T$ dans $\ppf_S$, et $E$ dans $\et_S$,
$$
\Hom_{\et_S}(\pi^s(T), E) \quad \iso \quad  \Hom_{\ppf_S}(T, \iota_{S}(E)) .
$$
Autrement dit, ce foncteur  associe à tout $T$ dans $\ppf_S$ un schéma $\pi^s(T)$ dans $\et_S$ et un morphisme
$$
h_{T}: T \rrr \pi^s(T) ,
$$ 
qui sont ``universels'' pour les morphismes de $T$ vers un $S$-schéma étale quasi-compact et séparé. 
L'exposant $s$ rappelle que le foncteur $\pi^s$ aboutit dans la catégorie des étales {\it séparés}.
\end{para}

\begin{prop}\label{sur} Considérons deux morphismes de schémas $T \by h E \by g S$ tels que $g$ soit étale quasi-compact et séparé, et que le composé $gh: T \to S$ soit plat de présentation finie. Alors 
\begin{itemize}
\item (i) Le morphisme $h$ est plat de présentation finie, et donc universellement ouvert (\ref{pr1.2});
\item (ii) Si $h$ est surjectif, alors le morphisme $h$ un épimorphisme  effectif universel; cela signifie que le diagramme 
$$
\xymatrix{T\times_{E}T  \ar@<0.5ex>[r] \ar@<-0.5ex>[r] &T \ \ar[r]^h & E} 
$$
est exact dans la catégorie des schémas, et le reste après tout changement de base $F\to E$;
\item (iii) le morphisme \og universel\fg\,   $h_{T}: T \rrr \pi^s(T)$, lorsqu'il existe, est fidèlement plat de présentation finie.
\end{itemize}
\end{prop}
\begin{proof} $(i)$ Puisque le morphisme $g$ est étale, quasi-compact et séparé, son morphisme diagonal $\Delta_{g}$ est une immersion ouverte et fermée, donc un morphisme plat de présentation finie, tout comme $gh: T \to S$; le lemme \ref{l1.2} entraîne que $h$ est plat de présentation finie, donc universellement ouvert (\ref{pr1.2}).

$(ii)$ Par définition, un morphisme de présentation finie est quasi-compact \ref{p.finie}; si $h$ est surjectif alors, d'après $(i)$, il est en particulier fidèlement plat quasi-compact, et on peut appliquer \cite[VIII, 5.3]{SGA1}.

$(iii)$ D'après $(i)$, le morphisme universel  $h_{T}$ est ouvert; son image $E \subset \pi^s(T)$ est donc  un ouvert, qui est quasi compact sur $S$ (puisque $T$ l'est); c'est donc un objet de $\et_{S}$; la propriété universelle implique que $E=\pi^s(T)$, donc que $h_{T}$ est surjectif.
\end{proof}

\begin{defn}\label{fact} Pour un morphisme plat de présentation finie  $f : T \to S$, la {\bf catégorie des factorisations} (sous-entendu : par un étale séparé) est la catégorie $\e(T/S)$ dont les objets sont les  factorisations de $f$ en   
$$
T \stackrel{h}{\rrr} E   \stackrel{g}{\rrr}  S,
$$
o\`u  $g$ est \'etale  quasi-compact et \emph{s\'epar\'e}, et où $h$ est surjectif. 

Dans la  catégorie $\e(T/S)$ une fl\`eche de $(g, h)$ vers $(g', h')$ est un morphisme de sch\'emas
$u : E \rrr E'$ tel que $h' = uh$ et $g'u = g$.  Comme $h$ est un épimorphisme (\ref{sur}, $(ii)$), un tel morphisme $u$, s'il existe, est unique. De plus, pour deux factorisations $T\by h E \to S$  et $T\by {h'} E' \to S$,  il en existe une qui les \og coiffe\fg: son schéma $E''$ est l'ouvert image du morphisme (ouvert)  $T \to E\times_{S}E'$ déduit de $h$ et $h'$.\end{defn}

\begin{para}
L'ensemble  $\overline{\e}(T/S)$ des classes $\overline{(g, h)}$ d'isomorphismes d'objets de $\e(T/S)$ est ordonné par la relation $\overline{(g, h)} \,\leq \,\overline{(g', h')}$ s'il existe une flèche, dans $\e(T/S)$, de $(g, h)$ vers $(g', h')$; d'après ce qui précède, cet ensemble ordonné est filtrant à gauche
\end{para}

\begin{prop}\label{p3.1}
Pour $T\in \ppf_S$, les conditions suivantes sont équivalentes:
\begin{thlist}
\item L'adjoint à gauche $F$ de l'inclusion $\iota_{S}: \et_S\to \ppf_S$ est défini \mbox{en $T$.}
\item La catégorie $\e(T/S)$ admet un objet initial.
\item L'ensemble ordonné $\overline{\e}(T/S)$ admet un plus petit élément.
\end{thlist}
La dernière condition est vérifiée lorsque  $\overline{\e}(T/S)$ est fini.
\end{prop}

\begin{proof} Il est clair que si l'adjoint $F$ est défini en $T$, alors $F(T)$ est initial parmi les factorisations de $T/S$. Réciproquement,  supposons que $\e(T/S)$ admette un élément initial $T \by h E$. Soit $T \by{h'} E'$ un morphisme de $S$-schémas, où $E'$ est étale, séparé et quasi-compact sur $S$, mais où on ne suppose pas que $h'$ soit surjectif. Il s'agit de voir qu'il existe un unique morphisme $E \to E'$  dans ${\sf Et.sep}_{S}$, compatible à $h$ et $h'$.

Puisque $f$ et $\Delta_{E'/S}$ sont universellement ouverts, $h'$ l'est aussi (\ref{l1.2}). Posons $E'' = h'(T) \subset E'$; c'est un $S$-schéma étale quasi-compact, qui s'insère dans la factorisation $T \to E'' \to S$, laquelle est un élément de $\e(T/S)$ puisque $T\to E''$ est surjectif. Comme $E$ est l'élément initial, on a un unique morphisme  $E\to E''$, d'où, par composition, le morphisme cherché $E \to E'$.
\end{proof}

\subsection{Existence de l'adjoint: énoncé et démonstration}

\begin{thm} \phantomsection\label{T1} Soit $S$ un schéma dont l'ensemble des composantes irréductibles est fini (par exemple $S$ intègre ou noethérien).
Pour tout morphisme de schémas $f : T \rrr S$ plat et de présentation finie, il existe un $S$-schéma $\pi^s(T/S)$ étale quasi-compact et \emph{séparé}, et un morphisme de $S$-schémas $h_{T}: T \to \pi^s(T/S)$, fidèlement plat de présentation finie, et  qui est universel pour les morphismes vers des schémas étales quasi-compacts et séparés sur $S$.
Autrement dit, le foncteur d'inclusion $\iota_{S}: \et_{S} \to \ppf_{S}$(\ref{not})  admet $T \mapsto \pi^s(T/S)$ pour adjoint à gauche. 
\end{thm}

\begin{proof}\label{proof} En vertu de la proposition \ref{p3.1}, il suffit de montrer que l'ensemble ordonné $\overline{\e}(T/S)$ est fini. Considérons un élément de $\overline{\e}(T/S)$, représenté par la factorisation $T\to E \to S$, et le carré cartésien:
\[\begin{CD}
 T \times_E T @>d_E>> T \times_S T\\
@VVV @VV{h\times h}V\\
 E@>\Delta_g>>E \times_S E.
 \end{CD}\]

 Puisque $E/S$ est étale et séparé, le morphisme diagonal $\Delta_g$ est une immersion ouverte et fermée ; donc $d_E$ est une immersion ouverte et fermée, qu'on peut identifier au sous-schéma ouvert et fermé $\IM(d_E)$ \cite[4.2.1]{EGAI}. On associe ainsi à toute classe d'isomorphisme d'objet de $\e(T/S)$ un sous-schéma ouvert et fermé de $T \times_S T$ ; l'exactitude (\ref{sur}) de la suite 
 $$
\xymatrix{T\times_{E}T  \ar@<0.5ex>[r] \ar@<-0.5ex>[r] &T \ \ar[r]^h & E} 
$$
montre que ce procédé est injectif. Il reste donc à voir que cet ensemble de sous-schémas est fini ; c'est évident si $S$ est noethérien. Si on suppose seulement que les composantes irréductibles de $S$ sont en nombre fini, on invoque l'ensemble des points maximaux, comme suit.
 \end{proof}
 \begin{para}\label{s3.4.1}

Soit ${\sf M}(X)$ l'ensemble des points génériques des composantes irréductibles d'un schéma $X$, encore nommés {\bf les points maximaux} de $X$ \cite[$0_{I}$, 2.1.1]{EGAI}. Voici les principales propriétés que nous utiliserons.

\begin{lemme}\phantomsection\label{l1.3.1} Soient $U$ et $V$ des ouverts fermés d'un schéma $X$. Si ${\sf M}(U) = {\sf M}(V)$, alors $U = V$.
\end{lemme}
\begin{proof} En effet, tout point $x \in X$ est dans l'adhérence de ${\sf M}(X) \cap \s(\oo_{X, x})$.
\end{proof}

\begin{lemme}\phantomsection\label{l1.3}  Soit $f : X \to Y$ un morphisme de schémas.
\begin{thlist} 
\item Si le morphisme $f$ est plat, il induit une application  ${\sf M}(f) : {\sf M}(X) \to {\sf M}(Y)$.
\item Si $f$ est fidèlement plat, alors  ${\sf M}(f)$ est surjectif.
\item Si $f$ est plat et de présentation finie, alors les fibres de ${\sf M}(f)$ sont finies.  En particulier, si l'ensemble des composantes irréductibles de $Y$ est fini, il en est de m\^eme pour $X$.
\end{thlist}
\end{lemme}
\begin{proof}
\n (i) Soit $\xi$ un point maximal de $X$; puisque $f$ est plat, l'homomorphisme $\s(\oo_{X, \xi}) \to \s(\oo_{Y, f(\xi)}) $ est fidèlement plat et il est, en particulier, surjectif ; l'anneau local $\oo_{Y, f(\xi)}$ a donc un seul idéal premier, i.e. $f(\xi)$ est un point maximal de $Y$ ; ainsi, $f$ induit un morphisme ${\sf M}(X) \to {\sf M}(Y)$.\ 

\n (ii) Soit $\eta$ un point maximal de $Y$ ; le morphisme 
$$
f' : X' =X\times_{Y}\s(\oo_{Y, \eta})\to~\s(\oo_{Y, \eta})
$$ 
est encore fidèlement plat ; donc le schéma $X'$ est non vide, et l'image d'un point maximal $\xi'$ de $X'$ ne peut qu'être égale à $\eta$.\

\n (iii) Supposons  $f$ plat et de présentation finie. Considérons, de nouveau, le morphisme  $f': X' \to \s(\oo_{Y, \eta})$, pour $\eta \in {\sf M}(Y)$; comme $X'_{\rm red}$ est de type fini sur le corps $\kappa(\eta)$, c'est un schéma noethérien, donc ${\sf M}(X') = \m(X'_{\rm red})$ est fini; d'autre part, le monomorphisme $X'\to X$ est plat, d'où une injection ${\sf M}(X')\to {\sf M}(X)$ dont l'image s'identifie à la fibre en ${\sf M}(f)\iv(\eta)$. \
\end{proof}
\end{para}


\subsection{Fonctorialités}

\begin{par}

Soit $S$ un schéma dont l'ensemble des composantes irréductibles est fini, de sorte que le foncteur $T \mapsto \pi^s(T /S)$ existe sur $\ppf_S$. 
Précisons quelle est l'image d'une flèche par ce foncteur: soient $f: T\to S$ et $f':T'\to S$ deux objets de $\ppf_S$; une flèche  dans cette catégorie, est un morphisme de schémas $u: T'\to T$ tel que $f' = fu$; la propriété universelle de $h_{T'}$ montre qu'il existe un unique morphisme $\pi^s(u)$ pour lequel le carré suivant est commutatif.
$$
 \xymatrix{
 T'  \ar[d]_{h_{T'}} \ar[r]^u & T \ar[d]^{h_{T}}\\
 \pi^s(T')  \ar[r]_{\pi^s(u)} & \pi^s(T)}
 $$
On explicite ici quelques propriétés générales de ces morphismes $\pi^s(u)$, en les reliant  aux foncteurs $\e(u): \e(T/S) \to \e(T'/S)$, introduits plus bas.
 Le terme \og factorisation \fg\, a, ici encore,  le sens restreint donné dans \ref{fact}.
\end{par}

\begin{prop}\label{l3.1} Soit $u: T'\to T$ un morphisme dans $\ppf_{S}$. On suppose que les composantes irréductibles de $S$ sont en nombre fini, de sorte que le foncteur $\pi^s$ existe;
\begin{enumerate}
\item Le morphisme $u$ induit, entre les catégories de factorisation, un foncteur 
$$
\e(u): \e(T/S) \to \e(T'/S) ;
$$
\item le morphisme canonique $h_{T}: T \to \pi^s(T/S)$ induit une équivalence $\e(h_{T})$ entre les catégories de factorisation, et un isomorphisme
$$
\pi^s(h_{T}) : \pi^s(T/S) \, \rt \, \pi^s(\pi^s(T/S)/S);
$$
\item le morphisme $\pi^s(u)$ est un isomorphisme si et seulement si le foncteur $\e(u)$ est une équivalence de catégories.
\end{enumerate}
\end{prop}

On notera que le foncteur $\e(u)$ est défini sans hypothèse de finitude sur les points maximaux de $S$, ni  d'hypothèse de platitude sur $u$, et que le passage  $u \mapsto \e(u)$ est contravariant, alors que $u\mapsto \pi^s(u)$ est covariant. 

\begin{proof}
1). Définition de $\e(u)$: soit $T \by{h} E \by{g} S$ une factorisation de $f$, et considérons la décomposition de $f' = fu$ en:
$$
T' \by{u} T \by h E \by g S .
$$
Puisque $f'$ et $\Delta_{g}$ sont universellement ouverts, il en est de même de $hu$; le sous-schéma $E' = {\rm Im}(hu)$ est donc un ouvert de $E$, et il est quasi-compact sur $S$ puisque $T'$ l'est; on a donc le carré commutatif suivant  
$$
\xymatrix{T' \ar[r]^{u} \ar[d]_-{h'} & T \ar[d]^-{h}\\
E' \ar[r]_{u'_{E}} & E} 
$$
où $h'$ est surjectif et où $u'_{E}$ est une immersion ouverte. On obtient ainsi une factorisation de $f'$ en $T'\by{h'} E' \by{g'} S$, avec $g' = gu'_{E}$. Cette construction définit le foncteur $\e(u) : \e(T/S) \to \e(T'/S)$, que l'on écrira simplement, et par abus de notation,
$$
\e(u)(h) \, = \, h'.
$$
Dans cette formule le symbole $h$ (resp. $h'$) désigne l'objet $T \by{h} E \by{g} S$ de $\e(T/S)$ (resp. $T'\by{h'} E' \by{g'} S$ de $\e(T'/S)$).

 Le morphisme $\pi^s(u)$ s'insère dans le diagramme commutatif suivant, où $v$ est surjectif, et où $u'$ est une immersion ouverte.
 $$
 \xymatrix{
 T' \ar@/_2pc/[dd]_{\e(u)(h_{T})} \ar[d]^{h_{T'}} \ar[r]^u & T \ar[d]^{h_{T}}\\
 \pi^s(T') \ar[d]^v \ar[r]^{\pi^s(u)} & \pi^s(T) \ar@{=}[d]\\
 F \ar[r]_{u'} & \pi^s(T)}
 $$
 \bb
 
 2). Pour chaque $T$,  le morphisme $h_{T}: T \to \pi^s(T/S)$ induit, par définition, une équivalence de catégories
 $$
 \e(h_{T}): \e(\pi^s(T/S) / S)\,  \rt \, \e(T/S) .
 $$ 
 La définition de $\pi^s(h_{T})$, lue dans le diagramme 
 $$
 \xymatrix{T \ar[r]^{h_{T}} \ar[d]_{h_{T}}& \pi^s(T) \ar@{=}[d] \\
 \pi^s(T) \ar[r]_-{\pi^s(h_{T})} & \pi^s(\pi^s(T)) ,} 
 $$
implique immédiatement que $\pi^s(h_{T})$ est un isomorphisme.
\bb

3). Vérifions d'abord la commutativité du diagramme 
 $$
 \xymatrix{\e(\pi^s(T/S) / S) \ar[d]_{\e({\pi^s(u)})} \ar[r]^-{\e(h_{T})} &\e(T/S) \ar[d]^{\e(u)}\\
 \e(\pi^s(T'/S) / S)  \ar[r]_-{\e(h_{T'})} &\e(T'/S)} \leqno{(*)}
 $$ 
 Explicitons le foncteur $\e(h_{T'})\e(\pi^s(u))$: soit $h: \pi^s(T) \to E$ un objet de $\e(\pi^s(T/S) / S)$, c'est-à-dire un morphisme surjectif vers un $S$-schéma $E$ étale quasi-compact et séparé; son image par $\pi^s(u)$ est le morphisme $h'$ du carré commutatif  $$
 \xymatrix{\pi^s(T') \ar[r]^{\pi^s(u)} \ar[d]_{h'}& \pi^s(T) \ar[d]^h\\
 E' \ar[r]_{u''} & E}
 $$
 où $h'$ est surjectif, et où $u''$ est une immersion ouverte. Son image dans $\e(T'/S)$ est le morphisme composé $h'h_{T'}$.
 
 Explicitons $\e(u)\e(h_{T})$: on a $\e(h_{T})(h) = hh_{T}: T \to E$, et  $\e(u)(h h_{T})= h'h_{T'}$, puisque $h'h_{T'}$ est surjectif et que $u''$ est une immersion ouverte. Le diagramme est donc commutatif.

Dans le diagramme ($\star$) les foncteurs \og horizontaux\fg\, sont des équivalence de catégories; cela montre que si $\pi^s(u)$ est un isomorphisme alors $\e(u)$ est une équivalence; réciproquement, si $\e(u)$ est une équivalence, la commutativité montre que $\e(\pi^s(u))$ est une équivalence; on est donc ramené à vérifier le lemme suivant, appliqué à  $\pi^s(u): \pi^s(T') \to \pi^s(T)$, à la place de  \mbox{$w: F' \to F$}. \end{proof}
\begin{lemme} Soit $w: F' \to F$ un morphisme de $S$-schémas étales quasi-compacts et séparés; on suppose que le foncteur $\e(w)$ est une équivalence de catégories. Alors $w$ est un isomorphisme.
\end{lemme}
\begin{proof} Puisque le foncteur $\e(w): \e(F) \to \e(F')$ est essentiellement surjectif, l'objet $(F' = F') \in \e(F')$ est l'image d'un élément $(h: F\to E) \in \e(F)$; le composé $F' \by w F \by h E$ est donc un isomorphisme de $S$-schémas; par suite, $w$ est un monomorphisme. 

Soit maintenant $(h': F'\to E') \in \e(F')$ l'image de $(F = F) \in \e(F)$; on a donc le carré commutatif
$$
\xymatrix{F' \ar[d]_{h'} \ar[r]^w & F \ar@{=}[d]\\
E' \ar[r]_{w'} & F}
$$
On vient de voir que $w$ est un monomorphisme; $h'$ est donc aussi un monomorphisme, mais c'est un épimorphisme \ref{sur}; donc $h'$ est un isomorphisme; cela montre que les deux objets $h: F\to E$ et $F = F$ ont des images isomorphes  par $\e(w)$; comme ce foncteur est fidèle, on conclut que $h$ est un isomorphisme, donc finalement que $w$ est un isomorphisme.
\end{proof}

Le résultat suivant m'a été signalé par {\sc Bruno Kahn}.\

\begin{prop}\phantomsection\label{c3.1} Soit $f: T \to S$ un morphisme plat de présentation finie.
\begin{description}
\item [$i)$] Le foncteur
\[f^*:\et_S\to \et_{T}\]
admet un adjoint à gauche $f_!^\Et$, donné, pour $F \in  \et_{T}$, par la formule
\[f_!^\Et(F)=\pi^s(F/S)\]

\item [$ii)$] {\rm (Transitivité de $\pi^s$)} Soit $U\in \ppf_{T}$. Alors on a un isomorphisme canonique, et fonctoriel en $U$:
\[\pi^s(U/S)\simeq \pi^s(\pi^s(U/T)/S).\]
\end{description}
\end{prop}

\begin{proof} $i)$  Notons d'abord que le morphisme composé $F\to T \to S$ est plat et de présentation finie, de sorte que $\pi^s(F/S)$ est défini. Soit $E\in \et_S$. On a une suite d'isomorphismes
\[\Hom_{\et_S}(\pi^s(F/S),E)\simeq \Hom_{\ppf_{S}}(F,E)\simeq \Hom_{\et_{T}}(F,E\times_S T)\]
le premier par définition de $\pi^s$ et le second par la propriété universelle du produit fibré. Cela démontre $i)$, ces isomorphismes étant naturels en  $E$.

$ii)$  Le foncteur $f^*$ de changement de base $E \mapsto T\times_{S}E$ donne lieu au diagramme commutatif suivant
$$
\xymatrix{\et_S \ar[r]^{\iota_{S}} \ar[d]_{f^*} & \ppf_S \ar[d]^{f^*}\\
\et_{T} \ar[r]_{\iota_{T}} &\ppf_{T}}
$$
Passant aux adjoints à gauche, on obtient l'isomorphisme canonique
$$
{}^{\rm ad}\iota_{S} \small{\circ} {}^{\rm ad}f^* \simeq {}^{\rm ad}f^* \small{\circ} {}^{\rm ad}\iota_{T}
$$
Traduisons cette formule: les adjoints ${}^{\rm ad}\iota$ sont par définition les foncteurs $\pi^s$;  l'adjoint ${}^{\rm ad}f^*$ qui est écrit  à gauche porte sur les catégories $\ppf$, c'est donc la restriction des scalaires au sens naïf (oubli de la base); tandis que le foncteur ${}^{\rm ad}f^*$ qui figure à droite concerne les schémas étales; d'après $i)$, il est donc  isomorphe à $\pi^s(-/S)$; d'où l'isomorphisme:
$$
\pi^s(U/S)\simeq \pi^s(\pi^s(U/T)/S).
$$
\end{proof}

\subsection{Exemples}

\begin{para}\label{ex.corps} L'unicité de l'adjoint, lorsqu'il existe, montre que si $S$ est le spectre d'un corps $k$, alors, d'après \ref{p1}, on a un isomorphisme de foncteurs (en $T$)
$$
\pi^{s}(T/S) \simeq \pi_{0}(T/S) ,
$$
où, rappelons-le, $\pi_{0}(T/S)$ désigne le spectre de la clôture algébrique séparable de $k$ dans l'anneau $\Gamma(T, \oo_{T})$.

Sur un anneau de base local hensélien, on a un résultat semblable:

\begin{cor} \label{ex.hens}Soit $S$ un schéma local hensélien dont les composantes irréductibles sont en nombre fini. Soit $f : T \to S$ un morphisme fidèlement plat de présentation finie. On suppose que $T$ est connexe.  Alors le $S$-morphisme étale $g: \pi^s(T/S) \to S$ est \emph{fini et local}.
Le morphisme de faisceaux d'algèbres  $ g_{\star}(\oo_{\pi^s(T/S)}) \to f_{\star}(\oo_{T})$ permet d'identifier la source à la ``clôture entière étale'' de $\oo_{S}$ dans $f_{\star}(\oo_{T})$. \end{cor} 

\begin{proof}
 Posons $E = \pi^s(T/S)$; c'est un schéma quasi-fini et séparé sur $S$. Comme $f$ est surjectif, il existe $e \in E$ tel que le point $g(e) = s$ soit le point fermé de $S$; d'après l'implication  \cite[18.5.11, $a) \Rightarrow c)$]{EGAIV4},  le schéma $E$ est somme de deux schémas $E'$ et $E''$, tels que $E' = \s(\oo_{E, e})$ et que la restriction $g_{|E'}: E' \to S$ soit un morphisme fini. Mais $T$ est connexe et le morphisme $h : T \to E' \sqcup E''$ est surjectif; donc $E''$ est vide.
\end{proof}

\end{para}
\begin{para}
Morphismes étales.\label{ex.ét}

Soit $f: T \to S$ un morphisme étale quasi-compact. Si $f$ est séparé, le morphisme $h: T \to \pi^s(T/S)$ est un isomorphisme; mais cette évidence ne dit pas grand chose sur les schémas étales intermédiaires. Supposons, par exemple, que $f$ soit fini étale et qu'une des fibres fermées $T_{s}\to s$ ne soit pas connexe; le \og dédoublement de $s$ \fg\, rappelé en \ref{pr1.3}, permet de décomposer le morphisme $f$ en  $T\to F \to S$, où $F\to S$ est étale non séparé (si $s$ n'est pas ouvert dans $S$), mais où le morphisme $T\to F$ est étale, et séparé puisque $F$ est réunion de deux ouverts $F_{1}$ et $F_{2}$ qui sont isomorphes à $S$; on a  ici $T \simeq \pi^s(T/F) \simeq \pi^s(T/S)$ et $\pi^s(F/S) \simeq S$.
On verra plus bas (\ref{p.12}) que si $S$ est normal, alors le morphisme $h_{T}: T \to \pi^s(T/S)$ est un isomorphisme local.

\end{para}

 \begin{para}
 Morphismes propres. Le résultat qui suit m'a été signalé par  {\sc Bruno Kahn}.\label{ex.pro}\\
 
\begin{prop} 
Soient $S$ un sch\'ema noeth\'erien et $f : T \rrr S$ un morphisme propre et lisse. Dans la factorisation de Stein
$$
T\by{h} \Spec_S(f_{\star}(\oo_{T}))  \by{g} S,
$$
le morphisme $g$ est \'etale fini. Notons $A(T/S) = \Spec_S(f_{\star}(\oo_{T}))$ l'enveloppe $S$-affine de $T$. On a les propriétés suivantes : 
\begin{enumerate}[(i)]
\item Les fibres de $h$ sont g\'eom\'etriquement connexes;
\item le morphisme $h: T \to A(T/S)$ est universel pour les morphismes de $T$ vers un $S$-schéma étale de type fini, séparé ou non;
\item le morphisme canonique $a: \pi^s(T/S) \to A(T/S)$ est un isomorphisme;
\item la formation de l'enveloppe affine commute à tout changement de base $S' \to S$ : le morphisme $ A(S'\times_{S}T / S') \to S'\times_{S}A(T/S)$ est un isomorphisme;
\item comme en (iv), mais pour les morphismes où $S'$ est un point géométrique de $S$.
\end{enumerate}
\end{prop}

Voir aussi la remarque qui suit \ref{p4.2}.

\begin{proof}
On a numéroté ces propriétés pour pouvoir facilement indiquer les dépendances logiques entre elles. Pour une d\'emonstration du fait que $g$ soit étale, voir \cite[4.3.4, 7.8.7 et 7.8.10]{EGAIII}, ou bien \cite[X 1.2]{SGA1}, ou enfin \cite[8.5.16]{FAG}.
La propriété (i) est essentiellement le théorème de connexion de Zariski (\cite[4.3.4]{EGAIII}). \\

L'implication (i)$\Rightarrow$ (ii) est démontrée en  \ref{p1.1}.\\

Montrons (iii). Le morphisme $a : \pi^s(T/S) \to A(T/S)$ est défini par la propriété universelle de $\pi^s(T/S)$ puisque le schéma étale $A(T/S)$ est fini, donc séparé; d'autre part, le morphisme propre $f$ se factorise en $T\by{h'}\pi^s(T/S)\by{g'} S$, et $g'$ est séparé de type fini; comme $h'$ est surjectif \ref{sur}, $g'$ est propre \cite[5.4.3 (ii)]{EGAII}; il est donc fini \cite[8.11.1]{EGAIV3} et en particulier affine: par la propriété universelle de l'enveloppe affine $A(T/S)$, on obtient un $S$-morphisme en sens inverse $b: A(T/S)\to \pi^s(T/S)$. Les deux propriétés universelles montrent que $a$ et $b$ sont inverses l'un de l'autre. 

On aura remarqué que cette démonstration de (iii) n'utilise ni la propriété (i), ni (ii).\\

 La propriété (v) est démontrée en \cite[7.8.7]{EGAIII}. Elle implique la propriété (iv), selon laquelle le morphisme 
 $$
 u:  A(S'\times_{S}T / S') \to S'\times_{S}A(T/S)
  $$
   est un isomorphisme; en effet,  la source et le but de $u$ sont deux $S'$-schémas finis étales, et, par hypothèse, le morphisme $u$ induit un isomorphisme entre leurs fibres.\\

 Montrons enfin que la conjonction des propriétés  (v) et (iii) implique (i) : d'après l'implication $(iii) \Rightarrow (i)$ du lemme \ref{p1.1}, on peut supposer que $S$ est le spectre d'un corps algébriquement clos, et il faut vérifier la propriété (ii) de la proposition dans ce cas; mais si $S$ est le spectre d'un corps tout $S$-schéma étale est séparé, donc le morphisme $h': T \to \pi^s(T/S)$ est universel pour tous les étales; et d'après (iii) $h'$ est isomorphe à $h: T \to A(T/S)$. 
 
 Ainsi la \og propriété d'échange\fg \, (v) implique la propriété de connexion (i).
 \end{proof}

\end{para}

\begin{para}
Enfin,  lorsque $S$ est un trait et que le morphisme $f : T \to S$ est lisse, la factorisation $T \to E \to S$ de l'énoncé \ref{p2} est isomorphe à la factorisation  $T \to \pi^s(T/S) \to S$.
\end{para}

\begin{para} \label{gen} Le foncteur $\pi^s$ ne commute en général pas à la restriction aux ouverts de $S$, ni à la restriction à la fibre générique lorsque $S$ est intègre, comme l'exemple suivant le montre: l'homomorphisme ${\bf Z} \to {\bf Z}[\sqrt{5}]$ est fini libre et ramifié en $2$; il n'est donc pas étale bien que sa fibre générique le soit.

Cependant, on montrera en \ref{p11} que des propriétés de commutation à certains changements de base sont vérifiées si la base est normale, ou géométriquement unibranche, et que le morphisme est lisse.
\end{para}

\section{L'adjoint comme quotient}\label{s5}

\subsection{Quotients étales séparés}

Dans ce paragraphe, on interprète le morphisme universel $h_{T}: T \to \pi^s(T/S)$ comme le passage  au quotient de $T$ par une relation d'équivalence convenable, et nous en tirons les premières conséquences.\

\begin{thm}\phantomsection\label{t1} Soit $f : T \rrr S$ un morphisme plat et de présentation finie de schémas. Soit $d : R \rrr T\times_{S}T$ le graphe d'une relation d'équivalence, o\`u $d$ est une immersion ouverte et fermée. Alors le faisceau fppf quotient $T/R$ est représentable par un schéma $E$ quasi-compact étale et séparé sur $S$; enfin, le morphisme $R \to T\times_{E}T$ est un isomorphisme.\end{thm}

\begin{proof}\phantomsection\label{r2}
Un théorème de {\sc M. Artin} montre que $T/R$ est un espace algébrique \cite[10.4]{LMB00} ; puisque l'immersion $R \rrr T\times_{S}T$ est ouverte, cet espace algébrique est étale, et il est séparé puisque l'immersion est aussi fermée ; un autre théorème  \cite[A.2]{LMB00} implique qu'un espace algébrique étale et séparé est (représentable par) un schéma.
\end{proof}

Voir l'appendice pour des démonstrations plus élémentaires du théorème \ref{t1}.

\begin{thm}\phantomsection\label{p7.1} {\rm [Relations d'équivalence associées aux factorisations].} Soit $f : T \to S$ un morphisme plat de présentation finie. Considérons le foncteur \og carré fibré\fg
$$
\cf_{T/S}: \e(T/S) \; \to \;  {\sf R}(T/S) ,\quad  \{ T\rightarrow E \rightarrow S\} \; \longmapsto \; \{T\times_{E}T \rrr T\times_{S}T\} 
$$
de la catégorie des factorisations $\e(T/S)$, au sens de \ref{fact}, vers la catégorie ${\sf R}(T/S)$ des graphes de relations d'équivalence dans $T$ qui sont ouverts et fermés dans  $T\times_{S}T$; l'image d'une flèche $u : E \to E'$ de la première catégorie est l'immersion $T\times_{E}T \to T\times_{E'}T$ de la seconde.
Alors, ce foncteur est une équivalence de catégories.
\end{thm}

\begin{proof} On a déjà remarqué (\ref{sur}) que, pour une factorisation $T \by h E \to S$, le morphisme $h$ est un épimorphisme effectif, i.e. que la suite
$$
T\times_E T\rightrightarrows T\by{h} E
$$
est exacte dans la catégorie des schémas; en particulier, le morphisme $h$ définit $E$ comme le quotient de $T$ par la relation $T\times_{E}T \rrr T\times_{S}T$. Réciproquement, soit $R$ une relation à graphe ouvert et fermé; d'après le théorème \ref{t1} le faisceau quotient $E' = T/R$ est représentable par schéma étale et séparé,  et le morphisme $R \to T\times_{E'}T$ est un isomorphisme. Le foncteur $R \mapsto (T \to T/R \to S)$ est donc quasi-inverse de celui de l'énoncé. \end{proof}

\begin{cor}\phantomsection\label{cor1} La relation d'équivalence $R$ telle que le morphisme $T \to T/R$ soit universel pour les étales quasi-compacts et séparés est l'objet initial dans la catégorie (équivalente à un ensemble ordonné) ${\sf R}(T/S)$ des relations d'équivalence à graphe ouvert et fermé dans $T\times_{S}T$.
\end{cor}
\begin{proof} C'est le critère \ref{p3.1} transféré dans la catégorie des relations d'équivalence grâce au théorème précédent.
\end{proof}
\begin{rque}  La démonstration (\ref{proof}) de l'existence du foncteur adjoint utilise implicitement la pleine fidélité du foncteur de (\ref{p7.1}).
\end{rque}

\begin{exs}
{\bf (1) (Morphisme d'anneaux artiniens)}. \, Soit $A \to B$ un homomorphisme d'anneaux locaux artiniens, faisant de $B$ un $A$-module libre de type fini. Notons $f: T=\s(B) \to S=\s(A)$ le morphisme de schémas associé. On va expliciter, pour ce morphisme,  la relation d'équivalence \og initiale\fg (\ref{cor1}), et le schéma $\pi^s(T/S)$.

\begin{lemme}\label{of} Soient $X$ un schéma et $U$ une partie ouverte et fermée de l'espace sous-jacent à $X$. Notons $\wt{U} = (U,{ \oo_{X}}|U)$ le schéma induit par $X$ sur l'ouvert $U$, et $j: \wt{U} \to X$ l'immersion ouverte associée. Alors $j$ est aussi une immersion fermée.
\end{lemme}
\begin{proof} On peut invoquer \cite[4.2.2.]{EGAI}; on peut aussi considérer l'ouvert fermé $V$ complémentaire de $U$ et l'immersion ouverte $j': \wt{V} \to X$. La décomposition en somme directe  $X = \wt{U} \sqcup \wt{V}$ conduit à l'isomorphisme de $\oo_{X}$-algèbres $\oo_{X} \by{\sim} j_{*}(\oo_{\wt{U}}) \times j'_{*}(\oo_{\wt{V}})$; l'homomorphisme $\oo_{X}\to j_{*}(\oo_{\wt{U}})$ est donc surjectif.
\end{proof}
Revenons au morphisme $T\to S$ de schémas locaux artiniens. Comme le schéma $T\times_{S}T$ est fini discret, le sous-schéma diagonal $\Delta \subset T\times_{S}T$ est porté par un \emph{ensemble} ouvert et fermé; notons  $\wt{\Delta}\to T\times_{S}T$  l'immersion ouverte et fermée indiquée dans le lemme précédent. Vérifions que ce sous-schéma $\wt{\Delta}$ est le graphe de la relation d'équivalence \og initiale \fg; ce schéma est évidemment minimal parmi les sous-schémas ouverts fermés de $T\times_{S}T$ qui contiennent $\Delta$; il reste donc à voir que c'est le graphe d'une relation d'équivalence. Or, cette relation est clairement réflexive et symétrique; la transitivité se traduit par l'inclusion (cf. \ref{par3})
$$
{p'_{0}}^{-1}(\wt{\Delta}) \cap {p'_{2}}^{-1}(\wt{\Delta}) \subset {p'_{1}}^{-1}(\wt{\Delta}) .
$$
Le schéma ${p'_{i}}^{-1}(\wt{\Delta})$ est un sous-schéma ouvert fermé de $T^3$ d'espace topologique sous-jacent l'ouvert et fermé ${p'_{i}}^{-1}(\Delta)$; ce schéma est donc égal à $\wt{{p'_{i}}^{-1}(\Delta)}$; or, l'inclusion
$$
{p'_{0}}^{-1}(\Delta) \cap {p'_{2}}^{-1}(\Delta) \subset {p'_{1}}^{-1}(\Delta) 
$$
est vraie puisqu'elle traduit la transitivité de la relation d'égalité dans $T$.\qed\\

Le schéma quotient $E = T/\wt{\Delta}$ est donc isomorphe à l'enveloppe $S$-étale  $\pi^s(T/S)$, et on a un isomorphisme (\ref{t1})
$$
\wt{\Delta} \by{\sim} T\times_{E}T .
$$
Puisque $T$ est fini sur $S$, le morphisme diagonal $\Delta \to T\times_{S}T$ est une immersion fermée; il résulte alors de l'isomorphisme précédent que l'immersion fermée $\Delta \to T\times_{E}T$ est surjective, donc que le morphisme $T\to E$ est radiciel (\cite[3.7.1]{EGAI}); on retrouve ainsi, lorsque $T\to S$ est une extension finie de corps, la factorisation usuelle en radicielle et étale (\cite[V, \S7, Prop. 12]{A}).

Traduisons cela en termes des anneaux sous-jacents : soit $\mathfrak{n}$ l'idéal maximal de $B\otimes_{A}B$ image réciproque,  par l'homomorphisme $B\otimes_{A}B \to B$, de l'idéal maximal de $B$; alors le sous schéma ouvert et fermé $\wt{\Delta}$ de $T\times_{S}T = \s(B\otimes_{A}B)$ associé à l'ensemble ouvert et fermé $\{ \mathfrak{n} \}$  est le spectre de $(B\otimes_{A}B)_{\mathfrak{n}}$, et le morphisme 
$$\s((B\otimes_{A}B)_{\mathfrak{n}}) \to \s(B\otimes_{A}B)
$$ 
est le graphe de la relation \og initiale \fg ; l'enveloppe étale $E = \pi^s(T/S)$ est donc le spectre de l'anneau défini par l'exactitude de la suite
$$
\xymatrix{\Gamma(E, \oo_{E}) \ar[r] & B \ar@<0.5ex>[r]  \ar@<-0.5ex>[r] &  (B\otimes_{A}B)_{\mathfrak{n}}}
$$
 Lorsque $A$ et  $B$ sont des corps, cette description de la cl\^{o}ture séparable se trouve déjà dans \cite{Fév69}, où sa vérification repose sur la remarque suivante : soit $A(x) \subset B$ le sous-corps engendré par un élément $x$ du noyau ; alors l'homomorphisme $A(x)\otimes_{A}A(x) \to A(x)$ est plat, autrement dit $A(x)$ est étale sur $A$.
 
 Enfin, en \ref{p10}, ce résultat est étendu au cas où $A$ est un corps et où $B$ est une algèbre de type fini sur $A$; en effet, la partie multiplicative complémentaire de $\mathfrak{n}$  dans $B\times_{A}B$ est égale à l'ensemble des éléments de cet anneau dont l'image par le morphisme $m: B\times_{A}B \to B$ est inversible, et cette partie conduit au même anneau de fractions que la partie $1+ {\rm Ker}(m) \subset B\times_{A}B$.\\

{\bf (2)  (Retour sur les $k$-algèbres)}\, Dans ce numéro et dans le suivant, on  considère des algèbres de type fini sur un corps $k$, et on donne une  description de $\pi_{0}  (= \pi^{s})$ comme noyau d'une double flèche explicite.

\begin{lemme}\label{p9} Soit $T = \s(A)$ le spectre d'une $k$-algèbre de type fini. 
\begin{enumerate}
\item Soit $R \by d T\times_{k}T$ une immersion ouverte qui est le graphe d'une relation d'équivalence. Alors $d$ est aussi une immersion fermée, et le quotient $T/R$ est étale fini sur $k$.

\item Soit $U \subset T\times_{k}T$ un ouvert contenant la diagonale. Alors le sous-anneau de  $A$
$$
\xymatrix{{\rm Ker}(A \ar@<0.5ex>[r]  \ar@<-0.5ex>[r] &  \Gamma(U))}
$$
est une $k$-algèbre étale finie.
\end{enumerate}
\end{lemme}

\begin{proof} La démonstration de ce résultat m'a été communiquée par {\sc Raynaud}, en 1970 (apparemment non publié). 

(1). On peut supposer que $k$ est séparablement clos. Comme $S = \s(k)$ est réduit à un point, tout morphisme $X \to S$ est universellement ouvert \cite[2.4.9]{EGAIV2}. Montrons  que les classes d'équivalence ensemblistes selon $R$ sont des ouverts de $T$.  Pour $x$ dans $T$, le morphisme composé $ x \to T \to S$ est universellement ouvert. Par changement de base, on en tire que le morphisme vertical de gauche du diagramme
$$
\xymatrix{T\times_{S}x = p_{0}^{-1}(x) \ar[r] \ar[d] & x \ar[d] \\
T\times_{S}T \ar[r]^{p_{0}} \ar[d]_{p_{1}} & T \ar[d]\\
T \ar[r] & S}
$$
est ouvert ; la classe d'équivalence ensembliste de $x$, à savoir $p_{1}(R \cap p_{0}^{-1}(x))$ est donc ouverte.

Soit $T = V_{1} \sqcup \cdots \sqcup V_{n}$ la partition de $T$ en sous-schémas ouverts, donnée par les classes d'équivalences  ($T$ est noethérien) ; ce sont des $k$-schémas de type fini. Désignons par $d_{ij} : R_{ij} \to V_{i}\times_{S}V_{j}$ la restriction de l'immersion ouverte $d$ à l'ouvert  $V_{i}\times_{S}V_{j}$ ; pour $i \neq j$ on a $R_{ij} = \emptyset $ ; montrons que $R_{ii} \to V_{i}\times_{S}V_{i}$ est un isomorphisme. Notons $V'_{i}$ l'ensemble des points fermés de $V_{i}$ ; comme le corps $k$ est séparablement clos {\it l'ensemble} $(V_{i}\times_{S}V_{j})'$ des points fermés de ce produit fibré est identique à $V'_{i} \times V'_{j}$ (\cite[p. 452]{EGAI}) ; par suite, l'application d'ensembles $R'_{ii} \to V_{i}' \times V'_{i}$ est surjective, donc l'immersion ouverte $R_{ii} \to V_{i} \times_{S} V_{i}$ induit une surjection entre les ensembles des points fermés ; c'est un isomorphisme. 
On voit alors directement que le schéma quotient est isomorphe au schéma somme de $n$ copies de $S$; mais on peut aussi invoquer le théorème (\ref{t1}); en effet, ce qui a été dit des $R_{ij}$ montre que $d: R \to T\times_{S}T$ est aussi une immersion fermée, et ce théorème entraîne alors   que le schéma quotient $T/R$ est étale et que $R$ s'identifie à $T\times_{T/R}T$.\\

(2) Appliquons  la première partie à la relation d'équivalence $R$ engendrée par $U$ ; comme $U$ contient par hypothèse la diagonale, $R$ est le schéma induit sur l'ouvert réunion des $U_{n}$, où $U_{0}$ est la réunion de $U$ et de son image par symétrisation et où
$$
U_{n+1} = p_{1}'(p_{0}'^{-1}(U_{n}) \cap p_{2}'^{-1}(U_{n})) .
$$
Cette construction assure la ``transitivité'' de $R$ (cf. \ref{par3}).
D'après la partie (1), le schéma quotient $E = T/R$ est étale fini sur $S$, le morphisme $R \to T\times_{E}T$ est un isomorphisme et  la suite $\xymatrix{T\times_{E}T  \ar@<0.5ex>[r]  \ar@<-0.5ex>[r] &T \ar[r]& E}$ est exacte; on en tire que $E$ est le spectre de l'anneau
$$
 {\rm Ker}(\xymatrix{A\ar@<0.5ex>[r]  \ar@<-0.5ex>[r] &  \Gamma(R)}) .
$$
Il reste à vérifier l'égalité
$$
{\rm Ker}(\xymatrix{A\ar@<0.5ex>[r]  \ar@<-0.5ex>[r] &  \Gamma(U)}) = {\rm Ker}(\xymatrix{A\ar@<0.5ex>[r]  \ar@<-0.5ex>[r] &  \Gamma(R)}).
$$
Désignons par $B$ l'anneau de gauche, et posons $F = \s(B)$ ; l'inclusion $U \subset R$ entra\^{i}ne que  $B$ contient l'anneau de droite; par définition de $B$, les deux morphismes composés
$$
\xymatrix{U \ar[r] &  T\times_{S}T \ar@<0.5ex>[r]  \ar@<-0.5ex>[r] &T \ar[r]& F}
$$
sont égaux ; cela montre que $U$ est contenu dans le fermé $T\times_{F}T$ ; mais ce dernier est une relation d'équivalence ; on a donc l'inclusion
$$
R \subset T\times_{F}T,
$$
autrement dit les deux morphismes composés $\xymatrix{R  \ar@<0.5ex>[r]  \ar@<-0.5ex>[r] &T \ar[r] & F}$ sont égaux, et cela montre que $B$ est contenu dans l'anneau de droite. 
\end{proof}

\begin{prop}\label{p10} Soient $A$ une $k$-algèbre de type fini et $J$ l'idéal noyau de l'homomorphisme $m : A\otimes_{k}A \to A$. Alors la cl\^{o}ture algébrique séparable de $k$ dans $A$ est l'anneau
$$
{\rm Ker}(\xymatrix{A\ar@<0.5ex>[r]  \ar@<-0.5ex>[r] &  (1+J)^{-1}A\otimes_{k}A})
$$
\end{prop}

\begin{proof} Notons $K$ l'anneau noyau de la double flèche, et montrons d'abord que tout élément $x \in A$ qui est algébrique et séparable sur $k$ est dans $K$. Pour simplifier la vérification, on élargit la partie multiplicative $1+ J$  en  $\Sigma = \{ t \in  A\otimes_{k}A  \mid m(t) \,{\rm est\, inversible\, dans\, } A\}$ ; cela ne change pas l'anneau des fractions.

À un polyn\^{o}me  $p(X) \in k[X]$ on associe les polyn\^{o}mes $p_{1}(X, Y)$ et $p_{2}(X, Y)$ définis par les relations suivantes dans $k[X, Y]$
$$
p(X) - p(Y) = (X - Y)p_{1}(X, Y),\hspace{0.5cm} p_{1}(X, Y) - p_{1}(X, X) = (X - Y)p_{2}(X, Y) ,
$$
de sorte qu'on a la relation
$$
p(X) - p(Y) = (X - Y)(p_{1}(X, X) + (X - Y)p_{2}(X, Y)).
$$

Puisque l'élément $x \in A$ est algébrique et séparable, il est racine d'un polyn\^{o}me $p(X) \in k[X]$ tel que $p'(x)$ soit inversible dans $A$.
L'égalité précédente  donne, dans $A\otimes_{k}A$,
$$
0 = (x\otimes1 - 1\otimes x)[p_{1}(x\otimes 1, x\otimes 1) + (x\otimes1 - 1\otimes x)p_{2}(x\otimes1, 1\otimes x)].
$$
 Notons $y$ l'élément de la seconde parenthèse ; compte tenu de ce que $p_{1}(X, X) = p'(X)$, on a  $m(y) = p'(x)$, donc $y$ est dans  $\Sigma$ et cela montre que $x$ est dans le noyau $K$.

Pour vérifier l'inclusion dans l'autre sens, revenons à la partie multiplicative $1+J$. Pour $t \in J$, notons, faute de mieux, par $K\{t\}$ le noyau
$$
K\{t\} = {\rm Ker}(\xymatrix{A\ar@<0.5ex>[r]  \ar@<-0.5ex>[r] &  (A\otimes_{k}A)_{1+t}})
$$
Il est clair que l'ouvert $D(1+t)$ contient la diagonale de $\s(A\otimes_{k}A)$ ; le lemme  (\ref{p9}, (2)) montre donc que $K\{t\}$ est une $k$-algèbre finie étale, donc que $K\{t\} \subset K$ ; d'autre part, la définition des anneaux de fractions entra\^{i}ne que  pour tout $x \in K$, il existe $t \in J$ tel que l'on ait, dans $A\otimes_{k}A$ la relation
$$
0 = (x\otimes1 - 1\otimes x)(1+t) ,
$$
c'est-à-dire $x \in K\{t\}$.
\end{proof}
\end{exs}

\subsection{Revêtements étales galoisiens}\

L'énoncé qui suit  est bien connu (voir p.ex. \cite[5.3.8]{Sza09}), mais le théorème \ref{p7.1} permet d'en donner une démonstration essentiellement triviale. Inversement, on peut voir le théorème \ref{p7.1} comme la généralisation \og naturelle\fg\, de la correspondance de Galois, les relations d'équivalences à graphe ouvert  fermé rempla\c cant les sous-groupes du groupe de Galois; cette analogie est déjà évoquée par {\sc Grothendieck} en  1960 (TDTE III, p.212-03). 
Dans le contexte galoisien qui suit, il est nécessaire de se restreindre aux schémas intermédiaires qui sont \emph{séparés} sur la base. En effet,
on a vu en \ref{pr1.3} qu'un morphisme fini étale $f: T \to S$, possèdant une fibre fermée non connexe, admet une factorisation  $T\to F \to S$, où $F\to S$ est  un morphisme étale birationnel, universellement fermé, mais non séparé sur $S$, et donc non entier; un tel $F$ ne peut évidemment pas être associé à un sous-groupe.

\begin{prop}\phantomsection\label{p4.1} {\rm (Correspondance galoisienne)} Soient $G$ un groupe fini et $f: T \to S$ un morphisme étale fini galoisien de groupe $G$, de sorte qu'on a un isomorphisme de $S$-schémas
$$
G \times T \to T\times_{S}T.
$$ 
On suppose que $T$ est connexe.
Pour tout sous-groupe $H \subset G$, $T/H$ est un schéma étale séparé et les morphismes $T \to T/H \to S$ forment une factorisation  de $f$, au sens de \ref{s4.2}. Réciproquement, pour une telle factorisation $T \to E \to S$, il existe un sous-groupe $H$ de $G$ tel que le carré
$$
\xymatrix{H \times T \ar[d] \ar[r]^{\sim} & T\times_{E}T \ar[d]\\
G\times T \ar[r]^{\sim} & T\times_{S}T}
$$
soit cartésien. Cette correspondance entre sous-groupes et factorisations est bijective.
\end{prop}

\begin{proof}
Pour un sous-groupe $H \subset G$, le morphisme $H\times T \to T\times_{S}T$ est le graphe (ouvert et fermé) d'une relation d'équivalence, et le quotient par cette relation est isomorphe au faisceau $T/H$, qui est donc un schéma étale et séparé sur $S$.

Réciproquement, considérons une factorisation $T\to E \to S$ et la relations d'équivalence associée $R=T\times_{E}T \to T\times_{S}T$. \ À chaque élément $g \in G$ est associée l'immersion $i_{g}: T \to T\times_{S} T, \; t \mapsto (gt, t)$. Comme $R$ est un ouvert fermé de $T\times_{S}T$, $i_{g}^{-1}(R)$ est un ouvert fermé du schéma $T$, lequel est supposé connexe. Notons $H \subset G$ l'ensemble des $g \in G$ tels que $i_{g}^{-1}(R) =T$; on a donc  $R = H\times T$.Vérifions, en termes ensemblistes, que $H$ est un sous-groupe de $G$: un élément $h$ est dans $G$ si et seulement si on a $(ht, t) \in R$ pour tout $t \in T$; en particulier, pour $h$ et $h'$ dans $H$, on a, pour tout $t \in T, (ht, t) \in R$ et $(h't, t) \in R$ ; d'où l'on tire pour tout $t \in T$, $(hh't, h't) \in R$ et $(h't, t) \in R$ ; la transitivité de $R$ entraîne que $hh'$ est dans $H$; $H$ étant fini, c'est bien un sous-groupe de $G$.
\end{proof}

\section{Des équivalences entre catégories de factorisation}

 Soient $f' : T' \to S$ et $f: T \to S$ des morphismes plats de présentation finie; un morphisme $u:T' \to T$  dans $\ppf_{S}$ induit un foncteur entre les catégories de factorisaton \ref{fact}
 $$
 \e(u) : \e(T/S) \; \to \; \e(T'/S), 
 $$
 et, si le foncteur adjoint existe, un $S$-morphisme
 $$
 \pi^s(u) : \pi^s(T'/S) \, \to \, \pi^s(T/S).
$$
 
On va dégager des propriétés de $u$ qui impliqueront que $\e(u)$ est une équivalence, et par suite (\ref{l3.1}), que $\pi^s(u)$ est un isomorphisme.

\subsection{\label{eq.0}} Rappelons (\ref{p7.1}) que le foncteur \og carré fibré \fg
$$
\cf_{T/S} : \e(T/S) \to \mathsf{R}(T/S)
$$
qui associe à la factorisation $T\to E \to S$ l'ouvert fermé   $T\times_{E}T \subset T\times_{S}T$,  est une équivalence de catégories.\\

Introduisons, comme dans \cite[18.5.3]{EGAIV4}, l'ensemble $\of(X)$ des ouverts fermés d'un schéma $X$; ce sont les sous-schémas induits sur les ouverts $U \subset X$, qui sont aussi (topologiquement) fermés; il s'avère (\ref{of}) que l'immersion ouverte  associée $(U, \oo_{X}|U)  \to (X, \oo_{X})$ est alors aussi une immersion fermée. Cet ensemble $\of(X)$ est ordonné par l'inclusion des sous-schémas, et on peut le considérer aussi comme une catégorie.\\

\n Le foncteur \og oubli de la structure de relation d'équivalence \fg\,  $\mathsf{R}(T/S) \to \of(T\times_{S}T)$ est pleinement fidèle.\\

\begin{lemme} Soient $f' : T' \to S$ et $f: T \to S$ des morphismes plats de présentation finie, et $u:T' \to T$ un $S$-morphisme de $\ppf_{S}$; désignons, pour simplifier, par  $\omega = u\times u$ le morphisme $T'\times_{S}T' \to T\times_{S}T$ induit par $u$, et par $\omega^*$ le foncteur image réciproque par $\omega$. Considérons le diagramme de foncteurs
$$
\xymatrix{\e(T/S) \ar[d]_{\e(u)} \ar[r]^{\cf_{T/S}} & \mathsf{R}(T/S) \ar[d]_{\mathsf{R}(u)} \ar[r] & \of(T\times_{S}T) \ar[d]^{\omega^*}\\
\e(T'/S) \ar[r]_{\cf_{T'/S}} & \mathsf{R}(T'/S)  \ar[r] & \of(T'\times_{S}T')  ,}
$$
où $\mathsf{R}(u)$ désigne l'induction habituelle pour les relations d'équivalence: $R \mapsto \omega^*(R)$ {\rm \cite[V, 3, {\bf a)}]{SGA3}}.
Alors les deux carrés sont commutatifs (à isomorphisme canonique près).
\end{lemme}

\begin{proof} 
En effet, soit $T\by h E \to S$ une factorisation, et soit $T'\by {h'} E' \to S$ son image par $\e(u)$ (\ref{l3.1}); alors dans le carré commutatif 
$$
\xymatrix{T' \ar[r]^{u} \ar[d]_{h'} & T \ar[d]^{h}\\
E' \ar[r]_{u'_{E}} & E ,} 
$$
 $h'$ est surjectif et $u'_{E}$ est une immersion ouverte, de sorte que le morphisme $R' = T'\times_{E'}T' \to T'\times_{E}T'$ est un isomorphisme;  de plus, dans le diagramme suivant
 $$
 \xymatrix{R' \ar[r]^-{\sim} & T'\times_{E}T' \ar[d] \ar[r] & T\times_{E}T\ar[d] \ar@{=}[r] &R\\
 & T'\times_{S}T' \ar[r]_{u\times u} &T\times_{S}T &}
 $$
 le carré est cartésien. Cela montre la commutativité du carré de gauche du diagramme de l'énoncé; la commutativité du carré de droite est évidente.
\end{proof}

\begin{prop}\label{eq.1} Gardons les hypothèses et les notations du lemme. On suppose que  $u$ est \scd, et que le morphisme $\omega^*$ induit, par image réciproque,  une équivalence entre les catégories d'ouverts fermés: 
$$
\omega^* :  \of(T\times_{S}T) \to  \of(T'\times_{S}T') .
$$
Alors les foncteurs $\mathsf{R}(u)$ et $\e(u)$ sont  des équivalences de catégories.\end{prop}

\begin{proof} 
Puisque le foncteur \og oubli \fg\, $\mathsf{R}(T/S) \to \of(T\times_{S}T)$ est pleinement fidèle, et que  $\omega^*$ induit une équivalence  sur les catégories d'ouverts fermés, il suffit de montrer que sa restriction $\sfR(u)$ aux relations d'équivalence est essentiellement surjective. Soit donc $R'$ une relation d'équivalence dans $T'$, dont le graphe est ouvert et fermé; par hypothèse, il existe un ouvert fermé  $R \in \of(T\times_{S}T)$ tel que $R' = \omega^*(R)$. Il faut montrer que $R$ est une relation d'équivalence dans $T$.\

 Comme d'habitude, nous écrirons les puissances d'un $S$-schéma en omettant l'indice $S$, donc  $T^2 = T\times_{S}T$ etc. Notons d'abord que la platitude de $f$ et celle de $f'$ entraînent que $\omega = u\times u$ est schématique dominant (\ref{s.schdom}), et idem pour $u\times u \times u$. 
 
 Pour un sous-schéma $Z$ de $T'^2$,  on note $\overline{Z}$ l'image schématique (\ref{im.sch}) du morphisme composé $Z \to T'^2 \by{\omega} T^2$. Si $U \subset T^2$ est ouvert fermé, alors le morphisme canonique
$$
\overline{\omega^*(U)} \to  U \leqno{(\star)}
$$
est un isomorphisme; en effet, le carré suivant est cartésien
$$
\xymatrix{{T'}^2 \ar[r]^{\omega} & T^2 \\
\omega^*(U) \ar[u] \ar[r]_-{\omega_{U}} & U \ar[u]} 
$$
 et $\omega_{U}$ est \scd, tout comme le morphisme $\omega$, puisque c'est la restriction de ce dernier à un ouvert; les images schématiques des morphismes verticaux sont donc égales.
 \medskip

Par hypothèse, le foncteur $\omega^*$ établit une équivalence entre les catégories d'ouverts fermés; la relation ($\star$) montre donc  que le foncteur $Z\mapsto \overline{Z}$ est un quasi-inverse de $\omega^*$; d'où l'on tire que le foncteur $Z \mapsto \omega^*(\overline{Z})$ est isomorphe à l'identité.  En particulier, pour des ouverts fermés $Z_{1}$ et $Z_{2}$ de $T'^2$, on a
$$
\overline{Z_{1}\cap Z_{2}} = \overline{Z_{1}}\cap\overline{ Z_{2}}.
$$

 Après ces préliminaires, reprenons l'ouvert fermé $R$ de $T^2$ tel que $\omega^*(R)$ soit une relation d'équivalence dans $T'$, et montrons que
 $R$ est une relation d'équivalence; il s'agit de vérifier les conditions (\ref{par1}), (\ref{par2}) et (\ref{par3}) de l'appendice qui expriment respectivement la réflexivité, la symétrie et la transitivité.\\

\itemize
\item (Réflexivité) Notons $\Delta'$  et $\Delta$ les sous-schémas diagonaux dans $T'^2$ et $T^2$; le morphisme $\Delta' \to \Delta$ induit par $\omega$ est égal à $u$; il est donc \scd. La commutativité du carré 
$$
\xymatrix{\Delta' \ar[r]^u \ar[d] & \Delta \ar[d]\\
T'^2 \ar[r]_{\omega} & T^2}
$$
implique alors que les images schématiques de $\Delta'$ et  $\Delta$ dans $T^2$ sont égales. Puisque la relation $\omega^*(R)$ est réflexive par hypothèse, on a $\Delta' \subset \omega^*(R)$; d'où l'on tire, en tenant compte de ($\star$), les inclusions suivantes.
$$
\Delta \subset \overline{\Delta} = \overline{\Delta'} \subset \overline{\omega^*(R)} = R.
 $$

 \item (Symétrie) Considérons l'automorphisme $\sigma$ de $T^2$ qui permute  les facteurs. L'inclusion $R \cap \sigma R \subset R$ d'ouverts fermés de $T^2$ devient une égalité par image réciproque par le morphisme $\omega$ qui est \scd. C'est donc une égalité. 
 \item (Transitivité) Il s'agit de vérifier l'inclusion
 $$
 {p'_{0}}^{-1}(R)  \cap  {p'_{2}}^{-1}(R)  \subset {p'_{1}}^{-1}(R) ,
 $$
 c'est-à-dire, d'après ($\star$), l'inclusion
  $$
 {p'_{0}}^{-1}(\overline{\omega^*(R)})  \cap  {p'_{2}}^{-1}(\overline{\omega^*(R)})  \subset {p'_{1}}^{-1}(\overline{\omega^*(R)}) .
 $$
 Puisque les morphismes de projection $p_{i}: T^{3} \to T^{2}$ sont plats, ils préservent la formation des images schématiques (\ref{im.sch}); on a donc ${p'_{i}}^{-1}(\overline{\omega^*(R)}) = \overline{{p'_{i}}^{-1}(\omega^*(R))}$.
 La formule  pour l'image schématique d'une intersection d'ouverts fermés, rappelée plus haut, et le fait que $\omega^*(R)$ soit transitive, permettent  de conclure.
\end{proof}


\subsection{Cas où les fibres de $u$ sont géométriquement connexes}

\begin{lemme}\label{l5.2} Soit $u: T'\to T$ un morphisme submersif de schémas {\rm \cite[3.10]{EGAI}}, dont les fibres sont connexes. Alors le foncteur $u^* : \of(T) \to \of(T')$ est une équivalence de catégories.
\end{lemme}
\begin{proof}Ce foncteur $u^*$ est essentiellement surjectif: en effet, soit $Z' \subset T'$ un ouvert fermé; alors on a l'égalité $u^{-1}(u(Z')) = Z'$; en effet,  pour $z' \in Z'$,  la fibre $u^{-1}(u(z'))$ est connexe par hypothèse, et elle rencontre (en $z'$) l'ouvert fermé $Z'$; elle est donc contenue dans $Z'$; il suffit donc de voir que $u(Z')$ est un ouvert fermé de $T$; mais $u$ est submersif et $Z' = u^{-1}(u(Z'))$ est un ouvert fermé de $T'$.

 La même égalité montre que le foncteur en question est pleinement fidèle.
\end{proof}

 \begin{prop} \label{inv.top}Soient $f:T\to S$  et  $f':T'\to S$ des morphismes plats et de présentation finie, et $u: T' \to T$ un $S$-morphisme \scd.  On suppose que le morphisme $u$ est universellement submersif et que ses fibres sont géométriquement connexes. Alors le foncteur
 $$
\e(u):  \e(T/S) \to \e(T'/S)
 $$
 est une équivalence de catégories.
 \end{prop}
\begin{proof}
Application directe de \ref{eq.1}, compte tenu du lemme ci-dessus.
\end{proof}

Rappelons que tout morphisme surjectif et ouvert, ou surjectif et fermé, est submersif.

Un homéomorphisme universel (= entier radiciel et surjectif) est évidemment universellement submersif et ses fibres sont géométriquement connexes; cela montre {\it l'invariance par homéomorphisme schématiquement dominant} du foncteur adjoint $T \mapsto \pi^s(T/S)$ lorsqu'il existe.

\subsection{Cas où $\oo_{T} \to u_{\star}(\oo_{T'})$ est bijectif}

 \begin{lemme} \label{stein} Soit $u: T' \to T$ un morphisme \qcqs de schémas, tel que l'application canonique 
 $$
 \oo_{T} \to u_{\star}(\oo_{T'})
 $$
 soit un isomorphisme. Alors l'image réciproque $Z \mapsto  u^*(Z) = T'\times_{T}Z$ définit une équivalence de catégories
 $$
 u^* : \of(T) \to \of(T').
 $$
 Un quasi-inverse de $u^*$ est le foncteur $\of(T') \to \of(T)$ qui associe à l'ouvert fermé $Z'$ de $T'$ l'image schématique (\ref{im.sch})  $\overline{Z'}$ du morphisme composé $Z' \to T' \to T$. Enfin, pour des ouverts fermés  $Z'_{1}$ et $Z'_{2}$ de $T'$, on a
 $$
 \overline{Z'_{1} \cap Z'_{2}} = \overline{Z'_{1}} \cap \overline{Z'_{2}}.
 $$
 \end{lemme}
 
 \begin{proof} Montrons que les deux applications indiquées sont réciproques l'une de l'autre. Considérons d'abord un ouvert fermé $Z$ de $T$, et l'ouvert fermé $T'\times_{T}Z$  de $T'$ qui lui correspond; pour vérifier l'égalité $\overline{T'\times_{T}Z} = Z$  considérons le carré cartésien suivant
 $$\xymatrix{T' \ar[r]^u & T\\
 T'\times_{T}Z \ar[u] \ar[r]_-{u_{Z}} & Z \ar[u]} 
 $$
 Le morphisme $u_{Z}: T'\times_{T}Z \to Z$ est la restriction de $u$ à l'ouvert $Z$; par suite il induit un isomorphisme
 $$
 \oo_{Z} \; \rt \;  {u_{Z}}_{\star}(\oo_{T'\times_{T}Z}) .
 $$
 Il est alors clair que le morphisme composé $T'\times_{T}Z \to T' \to T$ a  pour image schématique $Z$. 
 
 Considérons maintenant un ouvert fermé $Z'$ de $T'$, et montrons que son image schématique $\overline{Z'}$ est un ouvert fermé de $T$ et que le morphisme   $\theta': Z' \to \overline{Z'}\times_{T}T'$ est un isomorphisme. Notons $Z''$ l'ouvert fermé complémentaire de $Z'$ dans $T'$, et désignons par $u': Z' \to T$ et $u'': Z'' \to T$ les restrictions de $u$ à ces sous-schémas. Les isomorphismes de $\oo_{T}$-algèbres
 $$
 \oo_{T}\, \rt \, u_{\star}(\oo_{T'})\,  \rt\,  u'_{\star}(\oo_{Z'}) \times u''_{\star}(\oo_{Z''}) 
 $$ 
 montrent que le morphisme 
 $$
  j: \overline{Z'} \sqcup \overline{Z''} \, \to \, T
 $$
 est un isomorphisme. Considérons alors le diagramme suivant où le carré est cartésien:
 $$
 \xymatrix{Z' \sqcup Z'' \ar[ddr] \ar[dr]|{\theta'\sqcup \theta''} \ar[drr]^{j'}&&\\
&V \sqcup W \ar[d] \ar[r] _{\bar{j}}& T' \ar[d]^{u}\\
& \overline{Z'} \sqcup \overline{Z''} \ar[r]_{j} &T}
$$
Le morphisme  $\bar{j}$ est un isomorphisme puisque $j$ en est un; $j'$ est aussi un isomorphisme par hypothèse; donc $\theta'\sqcup \theta''$, et à fortiori $\theta'$, sont des isomorphismes.

Compte tenu de ce qui précède, l'égalité $\overline{Z'_{1} \cap Z'_{2}} = \overline{Z'_{1}} \cap \overline{Z'_{2}}$ est évidente par image réciproque de $T$ à $T'$.
 \end{proof}

 \begin{prop} \phantomsection\label{p4.2} Soient $f: T\to S$ et $f': T'\to S$ des morphismes plats et de présentation finie, et   $u: T' \to T$ un $S$-morphisme.
 On suppose que l'application canonique $\oo_{T} \rrr u_{\ast}(\oo_{T'})$ est bijective. Alors le foncteur induit par $u$ (\ref{l3.1})
$$
\e(u): \e(T/S) \to\e(T'/S),  
$$
est une équivalence de catégories.
\end{prop}

\begin{proof} C'est une conséquence immédiate de \ref{eq.1} et du lemme précédent (\ref{stein}).
\end{proof}
 
 \begin{rque}
 On peut rapprocher ce résultat de la factorisation de Stein rappelée en (\ref{ex.pro}): en effet, soit $f: T \to S$ un morphisme propre et lisse, et soit 
 $$
 T \by{h}  \Spec_{S}(f_{\star}(\oo_{T})) \by g S 
$$
sa factorisation de Stein; alors $g$ est étale fini, donc séparé, et en posant $E =  \Spec_{S}(f_{\star}(\oo_{T}))$, le morphisme $h$ est surjectif et  $\oo_{E} \to h_{\star}(\oo_{T})$ est bijectif. La proposition qui précède redonne l'isomorphisme 
$$
\pi^s(h): \pi^s(T/S)\,  \rt \, \pi^s(E/S) = \Spec_{S}(f_{\star}(\oo_{T})).
$$
\end{rque}
\section{Cas d'une base normale. Prolongements}

Soit $T \to S$ un morphisme lisse de présentation finie, où $S$ est un schéma normal dont les composantes irréductibles sont en nombre fini, de sorte que l'adjoint existe. On va montrer les deux propriétés de prolongement suivantes.

1). Soit $u: T' \to T$ une immersion ouverte schématiquement dense. Alors le morphisme $\pi^s(u): \pi^s(T'/S) \to \pi^s(T/S)$ est un isomorphisme.

2). On suppose de plus que $S$ est intègre, de point générique $\xi$. Alors, une  factorisation \og générique\fg\,  $T_{\xi} \to E_{0} \to \xi$ se prolonge de fa\c con unique en une factorisation $T \to E \to S$ (\ref{fact}), dont la fibre en $\xi$ est isomorphe à la factorisation générique donnée. 

Dans le langage de \ref{eq.0}, le morphisme  \og réduction à la fibre générique \fg\, $T_{\xi} \to T$ établit une équivalence de catégories de factorisations
$$
\e(T/S) \to \e(T_{\xi}/S) .
$$

Remarquons que la démonstration de \ref{p2} fournit un tel prolongement (pour une base de Dedekind), mais  par des méthodes qui semblent trop liées à la dimension 1 pour  pouvoir être généralisées telles quelles en plus grande dimension.

\subsection{}\label{s6}
Rappelons qu'un schéma $X$ est dit {\it normal} si tous ses anneaux locaux $\oo_{X, x}$ sont intègres et intégralement clos (\cite[5.13.5]{EGAIV2}). 

Dans la suite, on a besoin d'une propriété du  morphisme $f : T \to S$ qui assure qu'une hypothèse de  normalité sur $S$ se propage à $T$ et à $T\times_{S}T$. C'est le cas si $f$ est lisse de présentation finie, ou plus généralement si $f$ est plat à fibres géométriquement normales \cite[11.3.13, (ii)]{EGAIV3}. Nous nous limiterons aux morphismes lisses et donc à l'adjoint à gauche du foncteur d'inclusion
$$
\iota_{S} :\et_{S} \to \sm_{S},
$$ 
 où $\sm_{S}$ désigne la catégorie des schémas lisses de présentation finie sur $S$.
 
 La normalité elle-m\^{e}me n'intervient dans la suite que par la propriété topologique suivante.

\begin{lemme}\phantomsection\label{l7} Soit $X$ un schéma tel que le spectre $\s(\oo_{X,x})$ de chacun de ses anneaux locaux soit intègre (c'est le cas si $X$ est normal). On suppose que l'ensemble des composantes irréductibles de $X$ est fini. Alors :
\begin{enumerate}
\item Les composantes irréductibles sont ouvertes, et en particulier l'adhérence schématique d'un ouvert de $X$ est une partie ouverte et  fermée de $X$;
\item soient $U$ et $V$ deux ouverts de $X$; notons $\overline{U}$, $\overline{V}$  et $\overline{U\cap V}$ les adhérences schématiques. Alors on a
$$
\overline{U\cap V} = \overline{U} \cap \overline{V};
$$
\item si $U \subset X$ un ouvert schématiquement dense, alors l'application $Z \mapsto Z\cap U$ définit une équivalence de catégories.
$$
\of(X) \rt \of(U).
$$
\end{enumerate}
\end{lemme}

\begin{proof} 1) Les composantes irréductibles de $X$ sont deux à deux disjointes, et sont en nombre fini ; elles sont donc ouvertes. 

2) Un ouvert contenant un point $x$ contient un ouvert irréductible contenant $x$, à savoir l'intersection de cet ouvert avec l'unique composante irréductible de $X$ contenant $x$; de plus, un ouvert irréductible contenant un point $x \in \overline{U} \cap \overline{V}$ rencontre $U$ et $V$, donc aussi leur intersection puisqu'il est irréductible; on a donc $x \in \overline{U\cap V}$. L'inclusion dans l'autre sens est claire.

3) D'après $1)$, l'adhérence (dans $X$), $V \mapsto \overline{V}$ définit une application $\of(U) \to \of(X)$; elle est réciproque de l'application de l'énoncé. En effet, si $V \in \of(U)$, alors $V = U \cap \overline{V}$ puisque $V$ est un fermé de $U$; de plus, l'application $\of(X) \to \of(X)$ définie par 
$Z \mapsto Z\cap U \mapsto \overline{Z\cap U}  = \overline{Z} \cap \overline{U}$, est l'application identique  puisque $Z$ est férmé dans $X$ et que $\overline{U} = X$ par hypothèse.\end{proof}

\begin{prop}\label{pur} Soient $S$ un schéma normal dont les composantes irréductibles sont en nombre fini, et $f:T \to S$ un morphisme lisse de présentation finie. Soit $u: T' \to T$ une immersion ouverte \scd e. Alors le foncteur
$$
\e(u) : \e(T/S) \to \e(T'/S)
$$
est une équivalence de catégories.
\end{prop}
\begin{proof}
 Les hypothèses de \ref{eq.1} sont vérifiées: puisque $S$ est normal et que $f$ est lisse le schéma $T\times_{S}T$ est normal; ses composantes irréductibles sont en nombre fini d'après \ref{l1.3}, $iii$); enfin l'ouvert $T'\times_{S}T'  \subset T\times_{S}T$ est schématiquement dense en raison de la platitude de $f$ (\ref{s.schdom}). Donc, d'après \ref{l7}, 3), l'image réciproque par $u\times u$ induit une bijection 
 $$
 \of(T\times_{S}T) \; \rt \; \of(T'\times_{S}T') ,
$$
qui permet de conclure.
\end{proof}

\begin{thm}\phantomsection\label{p11} Soient $S$ un schéma intègre et normal, et $\xi$ son point générique. Soit $f : T \rrr S$  un morphisme lisse de présentation finie. Alors 
pour tout ouvert non vide $U$ de $S$, le morphisme canonique
$$
 \pi^s(U\times_{S}T/U) \to U\times_{S}\pi^s(T/S) 
$$
est un isomorphisme; autrement dit, le foncteur $T \longmapsto \pi^s(T/S)$, adjoint à gauche de  l'inclusion de catégories 
$$
\et_{S} \, \rrr \, \sm_{S}
$$
commute à la restriction aux ouverts de $S$.

De plus, on peut passer à la limite: la restriction à la fibre générique définit une équivalence de catégories
$$
\e(T/S)\;  \rt \; \e(T_{\xi}/ \xi),
$$
et on a un isomorphisme
$$
\pi^s(T_{\xi} / \xi)  \; \iso \; \pi^s(T/S)_{\xi},
$$
\end{thm}

\begin{proof} Soit $U$ un ouvert non vide de $S$; pour tout $S$-schéma $X$, on note $X_{U}$ l'image réciproque de $U$ dans $X$. L'image du morphisme composé $U\times_{S}T = T_{U} \to U \to S$ est contenue dans $U$, de sorte que toute factorisation, au sens de \ref{fact}, de ce morphisme $T_{U} \to S$ est une factorisation de $T_{U} \to U$. On voit de même que $\e(T_{\xi}/S) = \e(T_{\xi}/\xi)$. Dans la démonstration qui suit on écrira donc  $\e(T_{U})$  et $\e(T_{\xi})$ à la place de $\e(T_{U}/S)$  et $\e(T_{\xi}/S)$.

Notons $u: T_{U} \to T$ la projection. Le foncteur induit par $u$ (\ref{l3.1})
$$
\e(u): \e(T) \to \e(T_{U}) 
$$
est donné par la restriction au-dessus de $U$, $E \mapsto U\times_{S}E = E_{U}$, puisque dans le carré cartésien obtenu par restriction
$$
\xymatrix{T_{U} \ar[d]_{1\times h} \ar[r]^u & T \ar[d]^h\\
E_{U} \ar[r]_{u'} & E}
$$
le morphisme $1\times h$ est surjectif, tout comme $h$, et que $u'$ est une immersion ouverte.

L'ouvert $T_{U}$ est dense dans $T$ car $\xi \in U$, et parce que $T$ est plat sur $S$(\ref{pr1.1}). Puisque $f$ est lisse, le schéma $T\times_{S}T$ est normal et on peut utiliser le lemme (\ref{l7}, 3)): la restriction à $T_{U}$ induit une équivalence 
$$
\of(T\times_{S}T) \, \rt \, \of(T_{U}\times_{S}T_{U}) .
$$
La proposition \ref{eq.1} entraîne alors que le foncteur  \og restriction au dessus de $U$ \fg\, $\e(T) \to \e(T_{U})$ est une équivalence de catégories. Pour préciser le morphisme $\pi^s(u)$, reprenons le diagramme de \ref{l3.1}:
 $$
 \xymatrix{
 T_{U} \ar@/_2pc/[dd]_{1\times h_{T}=\e(u)(h_{T})} \ar[d]^{h_{T_{U}}} \ar[r]^u & T \ar[d]^{h_{T}}\\
 \pi^s(T_{U}) \ar[d]^v \ar[r]^{\pi^s(u)} & \pi^s(T) \ar@{=}[d]\\
  \pi^s(T)_{U} \ar[r]_{u'} & \pi^s(T)}
 $$
Puisque le foncteur $\e(u)$ est une équivalence de catégories il échange les éléments initiaux; le morphisme $v$ est donc un isomorphisme de $U$-schémas étales 
$$
\pi^s(T_{U})\;  \rt \; \pi^s(T)_{U}
$$
et $\pi^s(u)$ l'immersion ouverte $\pi^s(T_{U}) \rt \pi^s(T)_{U} \to \pi^s(T)$. 
\bb

Montrons maintenant qu'il existe un ouvert non vide $U \subset S$ tel que le changement de base associé au morphisme $j_{U}: \xi \to U$  induise une équivalence de catégories $\e(T_{U}) \to \e(T_{\xi})$. Comme, en général,  le morphisme $T_{\xi} \to S$ n'est pas de présentation finie, ni ouvert, on ne peut pas utiliser la définition \ref{l3.1} pour relier les catégories de factorisation $\e(T_{\xi})$  et $\e(T_{U})$; c'est pourquoi on passe ici par le changement de base $j_{U}^*$.

 Pour tout ouvert non vide $U$, donc contenant $\xi$, le foncteur $j_{U}^*: \e(T_{U}) \to \e(T_{\xi})$ est pleinement fidèle: en effet, pour $T = T_{U}$ ou $T = T_{\xi}$, les foncteurs $\e(T) \to \of(T\times_{S}T)$ sont pleinement fidèles (\S5); de plus, le foncteur
  $$
  (j_{U}\times j_{U})^*: \of(T_{U}\times_{S}T_{U}) \to \of(T_{\xi}\times_{S}T_{\xi})
  $$ est pleinement fidèle puisque $T_{\xi}\times_{S}T_{\xi}$  contient les points maximaux de $T_{U}\times_{S}T_{U}$ (\ref{l1.3.1}).
 
Considérons alors un objet de $\e(T_{\xi})$, c'est-à-dire une factorisation \og générique\fg, 
$$
T_{\xi} \to E_{0} \to \xi .
$$
Par passage à la limite (\cite[8.8.2]{EGAIV3}), il existe un ouvert $U$ de $S$ pour lequel il existe un diagramme $T_{U}\to F \to U$ prolongeant le diagramme générique, et on peut se ramener au cas où ce diagramme constitue une factorisation de $T_{U}\to U$(\cite[17.7.8,$ii)$]{EGAIV4}). Notant $\mathcal{U}$ l'ensemble des ouverts non vides de $S$, et $ \overline{\e}$ l'ensemble des classes d'isomorphie d'objets de $\e$, la limite inductive des changements de base $\xi \to U$ donne donc une bijection
$$
\underrightarrow{\rm lim}_{\, \mathcal{U}\, } \overline{\e}(T_{U}) \to \overline{\e}(T_{\xi}).
$$
Mais ces ensembles  $\overline{\e}$ sont finis; il existe donc un ouvert $U$ de $S$ tel que le foncteur
$$
j_{U}^*: \e(T_{U}) \to \e(T_{\xi})
$$
soit une équivalence de catégories; il échange donc les éléments initiaux de ces catégories, d'où l'on tire un isomorphisme
$$
\pi^s(T_{\xi}) \; \rt \; \pi^s(T_{U})_{\xi}.
$$
On a montré plus haut que le morphisme canonique $\pi^s(T_{U}) \to \pi^s(T)_{U}$ est un isomorphisme; comme $j_{U}^*(\pi^s(T)_{U}) = \pi^s(T)_{\xi}$, on obtient, par composition, l'isomorphisme annoncé
$$
\pi^s(T_{\xi}/\xi) \; \rt \; \pi^s(T/S)_{\xi}.
$$

\end{proof}

\begin{cor}[{\rm Compatibilité au produit}] \label{prd} Soit $S$ un sch\'ema normal  int\`egre de point g\'en\'erique $\xi$. Soient $f : T \rrr S$  et $f' : T' \rrr S$ deux morphismes  lisses et de pr\'esentation finie. Alors le morphisme canonique 
$$
\mu: \pi^s(T\times_{S}T' /S) \rrr  \pi^s(T/S)\times_{S}\pi^s(T'/S)
$$ 
est un isomorphisme.
\end{cor}
\begin{proof}
Le morphisme $\mu$ relie des $S$-schémas étales quasi-compacts et séparés;  il est surjectif puisque les morphismes $h_{T}$  et $h_{T'}$ le sont (\ref{sur}), et par suite aussi le morphisme composé 
$$
T\times_{S}T' \to \pi^s(T\times_{S}T' / S)\to \pi^s(T/S)\times\pi^s(T'/S).
$$ 
Il suffit donc (\ref{l1.1}) de voir que son morphisme générique 
$$
\pi^s(T\times_{S}T' / S)_{\xi} \, \rrr \, \pi^s(T/S)_{\xi}\times_{\xi}\pi^s(T'/S)_{\xi}
$$ 
est un isomorphisme. Suivant le théorème \ref{p11}, et \ref{ex.corps}, cela se réécrit
$$
\pi_{0}(T_{\xi}\times_{\xi}T'_{\xi} / \xi) \, \rrr \, \pi_{0}(T_{\xi}/\xi)\times_{\xi}\pi_{0}(T'_{\xi}/\xi) .
$$ 
 L'énoncé sur $S$ découle donc de la compatibilité au produit lorsque la base est un corps (\ref{p1} (iv)).
\end{proof}

\section{Changements  de base}
\subsection{\label{7.1}}
Soit $\phi : \wt{S} \to S$ un morphisme entre des schémas dont les composantes irréductibles sont en nombre fini, et soit $T\to S$ un morphisme plat et de présentation finie. D'après la propriété universelle de $\pi^s(\wt{S}\times_{S}T/\wt{S})$, on a  un morphisme de $\wt{S}$-schémas
$$
\pi^s(\wt{S}\times_{S}T/\wt{S}) \to \wt{S} \times_{S}\pi^s(T/S).
$$
On va dégager des conditions pour que ce morphisme soit un isomorphisme. Le paragraphe précédent montre que c'est le cas lorsque $S$ est un schéma normal intègre de point générique $\xi$, que $\phi$ est le morphisme $\xi \to S$ et enfin que le morphisme $T \to S$ est lisse de présentation finie.

En ($7.2$), les hypothèses portent sur les images réciproques par $\phi$ (supposé surjectif) des ouverts fermés du schéma $T\times_{S}T$, comme en \ref{eq.1}; mais ici les vérifications sont plus élémentaires. Ces hypothèses impliquent que l'image réciproque par $\phi$ induit une équivalence de catégories $\e(T/S) \by{\sim} \e(\wt{S}\times_{S}T/\wt{S})$ entre les catégories de factorisation, et que, par suite,  ce foncteur donne, pour leurs éléments initiaux $\pi^s$, l'isomorphisme cherché.

En ($7.3$), les schémas de base sont supposés normaux et, d'après \ref{p11}, la formation des $\pi^s$ commute à la localisation; cela permet de se ramener aux situations génériques et d'utiliser les résultats analogues sur des corps de base. 

On considère enfin le cas où $S$ est géométriquement unibranche; le morphisme de normalisation $\phi: \wt{S} \to S$ est alors un homéomorphisme universel, et, comme tel,  il induit un isomorphisme $\pi^s(\wt{T}/\wt{S}) \by{\sim} \wt{S}\times_{S}\pi^s(T/S)$; cela permet d'étendre, pour un morphisme lisse $T \to S$, des propriétés vraies sur une base normale. \\

Dans la suite, pour un $S$-schéma $X$, on notera $\wt{X} = \wt{S}\times_{S}X$ le schéma obtenu par le changement de base $\wt{S} \to S$, et même notation pour les morphismes.\\

\subsection{}

On utilisera les lemmes suivants.
 
 \begin{lemme}\label{l7.2.1} Soit $\phi : \tilde{S} \to S$ un morphisme surjectif  de schémas. Considérons deux morphismes de $S$-schémas $ X \stackrel{u}{\rrr} Y \stackrel{j}{\longleftarrow} U$, o\`u $j$ est une immersion ouverte. Si un morphisme $v' : \wt{X} \to \wt{U}$ est tel que $\wt{u} = \wt{j}v'$, alors il provient par changement de base d'un morphisme $v : X \to U$ tel que $u = vj$.
 \end{lemme}
 \begin{proof}
 Considérons le carré cartésien suivant:
 $$
 \xymatrix{X \ar[r]^u &Y\\
 X_{U}\ar[u]^{j_{X}} \ar[r] & U \ar[u]_{j}}
 $$
 Dans l'image réciproque de ce carré par $\phi: \wt{S} \to S$, l'application $\wt{j_{X}}$ a une section; elle est donc surjective; comme le morphisme $\phi$ est surjectif par hypothèse, $j_{X}$ est surjective; comme c'est aussi une immersion ouverte, c'est un isomorphisme.
 \end{proof}
 
 \begin{lemme}\label{l7.2.2} Soit $\phi : \tilde{S} \to S$ un morphisme surjectif  de schémas. Soit $T \to S$ un morphisme de schémas, et $R \subset T\times_{S}T$ un sous-schéma ouvert. Si $\wt{R}$ est le graphe d'une relation d'équivalence dans $\wt{T}$, alors $R$ est le graphe d'une relation d'équivalence dans  $T$.
 \end{lemme}
 \begin{proof}
 Il s'agit de vérifier les conditions (\ref{par1}), (\ref{par2}) et (\ref{par3}) de l'appendice qui expriment respectivement la réflexivité, la symétrie et la transitivité.\\

\itemize
\item (Réflexivité) L'existence du morphisme $s : T \to R$ qui factorise le morphisme diagonal (\ref{par1}), provient de l'application du lemme \ref{l7.2.1} au diagramme $T \stackrel{\Delta}{\rrr} T\times_{S}T \stackrel{d}{\longleftarrow} R$, puisque $d$ est une immersion ouverte.

 \item (Symétrie) Considérons l'automorphisme $\sigma$ de $T^2$ qui permute  les facteurs. L'inclusion $R \cap \sigma R \subset R$ d'ouverts fermés de $T^2$ devient une égalité par image réciproque par le morphisme $\phi$ qui est surjectif. C'est donc une égalité. 
 
 \item (Transitivité) Il s'agit de vérifier l'inclusion suivante d'ouverts du produit triple $T^3$:
 $$
 {p'_{0}}^{-1}(R)  \cap  {p'_{2}}^{-1}(R)  \subset {p'_{1}}^{-1}(R).
 $$
 On applique encore le lemme \ref{l7.2.1}, mais cette fois au diagramme 
 $$
 {p'_{0}}^{-1}(R)  \cap  {p'_{2}}^{-1}(R)  \to T^3 \stackrel{j}{\longleftarrow} {p'_{1}}^{-1}(R)
 $$ 
 où $j$ est l'image réciproque par $p'_{1}$ de l'immersion ouverte $R \subset T^2$.
 
\end{proof}

 \begin{prop}\label{ch.ba} Soit $\phi : \wt{S} \to S$ un morphisme surjectif  de schémas. On suppose que pour tout changement de base plat $S' \to S$,  l'application induite par $\phi^*$, $\of(S') \to \of(\wt{S}\times_{S}S')$ est bijective. Alors, pour un morphisme $T\to S$ plat et de présentation finie, le foncteur de changement de base $ E \mapsto \wt{S}\times_{S}E= \wt{E}$ induit une équivalence entre les catégories de factorisation
$$
\e(T/S) \to \e(\wt{T} / \wt{S}).
$$
\end{prop}

\begin{proof}
 
Soit $f: T\to S$ un morphisme plat de présentation finie. Pour montrer que le foncteur de changement de base induit une équivalence 
$$
\e(T/S) \to \e(\wt{T}/\wt{S})
$$
on utilise la bijection
$$
\of(T\times_{S}T) \to \of(\wt{T}\times_{\wt{S}}\wt{T}),
$$
induite par $\phi^*$, et on se ramène à vérifier ceci: soit $R \subset T\times_{S}T$ un sous-schéma ouvert et fermé tel que $\wt{R} \subset \wt{T}\times_{\wt{S}}\wt{T}$ soit le graphe d'une relation d'équivalence sur $\wt{T}$; alors $R$ est le graphe d'une relation d'équivalence sur $T$. Mais c'est exactement l'énoncé \ref{l7.2.2}.
 \end{proof}

 \begin{cor} \label{homeo}Soit $\phi: \wt{S} \to S$ un morphisme surjectif de schémas. On suppose que l'une des propriétés suivantes est vérifiée:
 \begin{itemize}
 \item  (a) $\phi$ est universellement submersif et ses fibres sont géométriquement connexes (c'est le cas si $\phi$ est un homéomorphisme universel);
 \item  (b) le morphisme $\phi$ est quasi-compact et quasi-séparé, et l'application $\oo_{S} \to \phi_{\star}(\oo_{\wt{S}})$ est bijective.
 \end{itemize}
 Alors, pour un morphisme $T\to S$ plat et de présentation finie, le foncteur de changement de base $ E \mapsto \tilde{S}\times_{S}E$ induit une équivalence entre les catégories de factorisation
$$
\e(T/S) \to \e(\tilde{T}/ \tilde{S}).
$$
En particulier, si on suppose de plus que l'ensemble des composantes irréductibles de $S$ est fini, et idem pour  $\wt{S}$, alors  le morphisme canonique 
 $$
 \pi^s(\wt{T}/\wt{S}) \to \wt{S}\times_{S}\pi^s(T/S)
 $$
 est un isomorphisme.
 \end{cor}
 
 Sous la première hypothèse, cet énoncé est à rapprocher de \cite[IX, 3.4]{SGA1}; la seconde hypothèse est vérifiée  par certains morphismes propres, et en particulier par certains éclatements.
 
 \begin{proof} Les hypothèses $(a)$ ou $(b)$ sont vérifiées par $\phi \times \phi$ puisque $T$ est plat sur $S$; le changement de base par $\phi \times \phi$  donne donc une bijection 
 $$
 \of(T\times_{S}T) \by{\sim} \of(\wt{T}\times_{\wt{S}}\wt{T}), 
 $$
 comme il découle du lemme \ref{l5.2} pour $(a)$, et du lemme \ref{stein} pour $(b)$. La proposition ci-dessus (\ref{ch.ba}) montre alors qu'on a une équivalence de catégories
$$
\e(T/S) \to \e(\wt{T} / \wt{S}).
$$ 
Cette équivalence échange les éléments initiaux, et elle est induite par le changement de base par $\phi$; d'où l'isomorphisme annoncé.
 \end{proof}

 
 \subsection{Bases normales, et  géométriquement unibranches}
 
 \begin{prop}{\label{7.3}} Soit $\phi: \wt{S} \to S$ un morphisme dominant de schémas normaux intègres. Soit $f: T\to S$ un morphisme lisse de présentation finie, et $\wt{f}: \wt{T} \to \wt{S}$ le morphisme obtenu par changement de base. Alors le morphisme canonique 
$$
\pi^s(\wt{T}/\wt{S}) \to \wt{S}\times_{S} \pi^s(T/S)
$$
est un isomorphisme.
\end{prop}
\begin{proof}
 Désignons par $\eta \to \xi$ le morphisme induit par $\phi$ sur les points génériques, de sorte qu'on a le carré commutatif suivant
 $$
\xymatrix{\eta \ar[r] \ar[d] & \wt{S} \ar[d]^{\phi}\\
\xi \ar[r] & S .} 
$$

Remarquons d'abord que le morphisme de $\wt{S}$-schémas 
$$
u: \pi^s(\wt{T} / \wt{S}) \to  \wt{S}\times_{S} \pi^s(T/S)
$$
 est étale quasi-compact séparé, et surjectif puisque  le morphisme $T \to \pi^s(T/S)$ est surjectif (\ref{sur}); il suffit donc, d'apr\`es le lemme \ref{l1.1}, de v\'erifier qu'il devient un isomorphisme  par le changement de base $\eta \to \wt{S}$. La commutativité du diagramme suivant montre que, pour voir que $u_{\eta}$ est un isomorphisme, il suffit de vérifier que les morphismes $\alpha, \beta$ et $\gamma$ sont des isomorphismes.
 $$
 \xymatrix{\pi^s(\wt{T}_{\eta}/\eta) \ar[d]_{\alpha} \ar@{=}^{\rm def.}[r] &\pi^s(\eta\times_{\xi}T_{\xi}/\eta) \ar[r]^{\gamma}& \eta\times_{\xi}\pi^s(T_{\xi}/\xi)\ar[d]^{\beta}\\
\pi^s(\wt{T}/\wt{S})_{\eta}\ar[r]_-{u_{\eta}} &\eta\times_{S}\pi^s(T/S) \ar@{=}_{\rm def.}[r] & \eta \times_{\xi}\pi^s(T/S)_{\xi}}
 $$
 Or, les flèches verticales $\alpha$ et $\beta$ sont des isomorphismes, d'après le théorème \ref{p11} appliqué respectivement aux morphismes $\wt{T} \to \wt{S}$, et  $T\to S$. Sur un corps le foncteur $\pi^s$ est isomorphe à $\pi_{0}$, et ce dernier commute  au changement de base $\eta \to \xi$ (\ref{p1} (ii)); on en tire que $\gamma$  est un isomorphisme. 
\end{proof}

 \begin{para}
 Rappelons \cite[0$_{IV}$, 23.2.1]{EGAIV1} qu'un anneau local $A$ est dit unibranche (resp. géométriquement unibranche) si $A_{\rm red}$ est intègre et si sa cl\^oture intégrale est locale (resp. locale à extension résiduelle radicielle).

 Un schéma est dit \emph{géométriquement unibranche} si chacun de ses anneaux locaux l'est \cite[ 6.15.1]{EGAIV2}.
Cette notion intervient dans la suite en raison du critère de \cite[ 6.15.3]{EGAIV2}:
 \medskip
 
 \n {\it Soit $X$ un schéma dont l'ensemble des composantes irréductible est fini. Soit $\varphi : \wt{X} \to X$ le morphisme associé au normalisé  
 $\wt{X}$ de $X_{\rm red}$. Alors $\varphi$ est un homéomorphisme universel si et seulement si $X$ est géométriquement unibranche.}

 Ces considérations montrent que l'énoncé \ref{p11} reste vrai si on y remplace l'hypothèse que $S$ est normal par l'hypothèse que $S$ est géométriquement unibranche. \end{para}
 
\begin{prop} Soit $S$ un sch\'ema géométriquement unibranche  intègre de point générique $\xi$. Soit $f : T \rrr S$  un morphisme lisse de présentation finie. Alors la restriction à la fibre générique définit une équivalence de catégories
$$
\e(T/S)\;  \rt \; \e(T_{\xi}/ \xi).
$$
En particulier, on a un isomorphisme
$$
\pi^s(T_{\xi} / \xi)  \; \iso \; \pi^s(T/S)_{\xi}  ,
$$
et pour tout ouvert non vide $U$ de $S$, le morphisme canonique
$$
 \pi^s(U\times_{S}T/U) \to U\times_{S}\pi^s(T/S) 
$$
est un isomorphisme; autrement dit, le foncteur $T \longmapsto \pi^s(T/S)$, adjoint à gauche de  l'inclusion de catégories 
$$
\et_{S} \, \rrr \, \sm_{S}
$$
commute à la restriction aux ouverts de $S$.
\end{prop}
 
 \begin{proof} Soit $\phi : \wt{S} \to S$ le normalisé de $S$. Puisque $T$ est lisse sur $S$, la fibre générique $T_{\xi}$  est un schéma normal, donc $\wt{T}_{\xi} = T_{\xi}$. Considérons le diagramme commutatif
 $$
 \xymatrix{\e(T) \ar[r]^{\phi^*} \ar[d]_{j^*}& \e(\wt{T}) \ar[d]^{\wt{j}^*}\\
 \e(T_{\xi}) \ar@{=}[r] & \e(\wt{T}_{\xi})}
 $$
 Par hypothèse, 
 $\phi$ est un homéomorphisme universel; on peut donc invoquer \ref{homeo},$(a)$ pour voir que le foncteur horizontal $\phi^*$ est une équivalence de catégories; comme $\wt{S}$ est normal et que $T$ est lisse sur $S$, le foncteur  $\wt{j}^*$ est une équivalence de catégories \ref{p11}. Le foncteur $j^*$ est donc lui aussi une équivalence.
 \end{proof}
 

\section{L'adjoint d'un étale}

\subsection{Adhérence d'une relation ouverte}

\begin{lemme}\label{adhR} Soit $f: T \to S$ un morphisme lisse de présentation finie. On suppose que $S$ est normal et intègre. Soit $R \subset T\times_{S}T$ un ouvert qui est le graphe d'une relation d'équivalence. Alors l'adhérence schématique $\overline{R}$ de $R$ dans $T\times_{S}T$ est un ouvert fermé qui est le graphe d'une relation d'équivalence.
\end{lemme}
\begin{proof}Comme il est rappelé en (\ref{s6}), le schéma $T^2 = T\times_{S}T$ est normal et ses composantes irréductibles sont en nombre fini; par suite (\ref{l7}) $\overline{R}$ est un sous-schéma ouvert et  fermé dans $T^2$; il définit une relation qui est clairement réflexive et symétrique; sa transitivité équivaut  (\ref{par3}) à l'inclusion
$$
{p'_{0}}\iv(\overline{R}) \cap {p'_{2}}\iv(\overline{R}) \subset {p'_{1}}\iv(\overline{R}) .
$$
Or, les projections $p'_{i}: T^{3}\to T^{2}$ sont des morphismes plats, ils préservent donc les adhérences schématiques (\ref{adh.sch}); d'où les égalités ${p'_{i}}\iv(\overline{R}) = \overline{{p'_{i}}\iv(R)}$;  pour conclure on utilise la compatibilité de l'adhérence aux intersections (\ref{l7}) et le fait que $R$ est une relation d'équivalence.
\end{proof}

\begin{prop}\phantomsection\label{p.12} {\rm  (L'adjoint d'un étale)} Soit $S$ un schéma intègre et normal, et $f : T \to S$ un morphisme étale de présentation finie. Alors le morphisme 
$$
h : T \to \pi^s(T/ S)
$$
est un isomorphisme local \cite[I, 4.4.2]{EGAI} ; plus précisément, $T$ admet un recouvrement par des ouverts $U$ sur lesquels le morphisme $h$ induit un isomorphisme de $U$ sur l'ouvert $h(U)$.

Enfin,  ce morphisme $h$ est aussi l'enveloppe séparée de $T \to S$ au sens suivant: tout morphisme de $S$-schémas $h': T\to F$, où $F$ est séparé sur $S$, est de la forme $h' = uh$ pour un unique morphisme $u: \pi^s(T/ S) \to F$ (On ne suppose pas que $F$ soit plat ni de type fini sur $S$).
\end{prop}

\begin{proof} Comme $f$ est étale, le morphisme diagonal $T \to T\times_{S}T$ est une immersion ouverte; le lemme qui précède (\ref{adhR}) montre que l'adhérence $R$ du sous-schéma diagonal est une partie ouverte et fermée de $T\times_{S}T$ qui est le graphe d'une relation d'équivalence. Elle est évidemment la relation, à graphe ouvert fermé, minimale; le quotient $T/R$ définit donc $\pi^s(T/ S)$.\

Montrons que le morphisme $h$ est un isomorphisme local. On peut recouvrir $T$ par des ouverts $U$ tels que $h(U)$ soit contenu dans un ouvert affine de $S$, donc tels que le morphisme $U\to S$ soit séparé. La relation induite par $R$ sur un tel ouvert est $R_{U} = R \cap (U\times_{S}U)$; c'est donc l'adhérence schématique dans  $U\times_{S}U$ du sous-schéma diagonal, lequel est fermé puisque $U$ est séparé sur $S$ ; donc $R_{U}$ est ce sous-schéma diagonal; ainsi, $R_{U}$ est la relation triviale sur $U$, et $h$ induit un isomorphisme $U \simeq h(U)$.

Montrons enfin que $h$ est l'enveloppe séparée de $f$. Posons $E = \pi^s(T/ S)$; on a vu au début de la démonstration que la relation $R = T\times_{E}T$ qui définit $E$ est l'adhérence schématique de la diagonale; cela implique que le morphisme $\Delta_{h}: T \to T\times_{E}T$ est \scd. Soit $h' : T \to F$ un morphisme vers un $S$-schéma séparé. Considérons le diagramme commutatif suivant.
$$
\xymatrix{T \ar[r]^{\Delta_{h}} \ar[d]_{\Delta_{h'}} & T\times_{E}T \ar@{-->}[dl]^w \ar[d]^{\phi}\\
T\times_{F}T \ar[r]_{\psi} & T\times_{S} T.}
$$
Les morphismes $\phi$ et $\psi$ sont des immersions ferm\'ees puisque $E$ et $F$ sont s\'epar\'es sur $S$, et on a vu que $\Delta_{h}$ est \scd\, ; on en tire  l'existence du morphisme $w$ rendant les triangles commutatifs. Consid\'erons maintenant le diagramme
$$
\xymatrix{T\times_{E}T \ar[d]_{w} \ar@<0.5ex>[r] \ar@<-0.5ex>[r] &T \ar@{=}[d] \ar[r]^h & E\ar@{-->}[d]^u\\
T\times_{F}T  \ar@<0.5ex>[r] \ar@<-0.5ex>[r] &T \ar[r]_{h'} & F .}
$$
La ligne sup\'erieure est exacte puisque $h$ est fid\`element plat quasi-compact, donc un \'epimorphisme effectif \cite[VIII 5.3]{SGA1} ; d'o\`u l'existence et l'unicit\'e de $u$. 
\end{proof}

 \subsection{L'espace des composantes connexes des fibres}\label{LMB} C'est le lieu ici d'évoquer rapidement l'espace algébrique $\pi_{0}$ qui représente les composantes connexes des fibres géométriques d'un morphisme lisse; cet espace algébrique est signalé dans le livre de {\sc Laumon -- Moret-Bailly} \cite[6.8]{LMB00} pour les schémas, et généralisé pour les champs par {\sc M. Romagny} \cite[6.2.6]{Rom11}. Voici l'énoncé pour les schémas.\\
 
 {\it Soit $f : T \rrr S$ un morphisme lisse de présentation finie de schémas. Alors il existe un \emph{espace alg\'ebrique} $\pi_{0}(T/S)$ qui est \'etale et quasi-compact sur $S$ et un morphisme $h : T \rrr \pi_{0}(T/S)$ ayant les propriétés suivantes:
 \begin{itemize}
 \item (a) Pour tout point géométrique $\xi$ de $S$, $\pi_{0}(T/S)(\xi)$ s'identifie à l'ensemble des composantes connexes de $T_{\xi}$;
 \item (b) les fibres de $h$ sont  g\'eom\'etriquement connexes (donc géométriquement irréductibles, puisque $h$ est lisse);
 \item (c) tout morphisme de $S$-schémas $T \to E$, o\`u $E$ est étale, se factorise par $h$;
 \item (d) la formation de $\pi_{0}(T/S)$ commute \`a tout changement de base sur $S$.
 \end{itemize} }
 
 Cet espace $\pi_{0}$ est construit comme le faisceau quotient $T/R$ pour la relation d'équivalence dont le graphe $R \subset T\times_{S}T$ est la réunion des composantes connexes des fibres de la première projection $p: T\times_{S}T \to T$, qui rencontrent la diagonale.
 
 Indiquons succinctement comment la propriété $(a)$ conduit à cette relation $R$ : soit $\xi$ un point géométrique de $S$; deux points $x$ et $y$ de $T_{\xi}$ sont dans la même composante connexe $C$ de $T_{\xi}$ si  on a $(x, y) \in C\times C \subset T_{\xi}\times_{\xi}T_{\xi}$; notant $C(x)$ la composante connexe de $x$, le graphe de la relation: \og être dans la même composante connexe\fg\, est donc 
 $$
 \bigcup_{x\in T_{\xi}} x \times C(x) ;
 $$
or, $x \times C(x)$ est la composante connexe de $p_{\xi}^{-1}(x) = x\times T_{\xi}$ qui rencontre la diagonale. Cette remarque et un peu de travail conduisent à la description donnée de la relation d'équivalence $R$.
 Le théorème \cite[15.6.5]{EGAIV3}  montre alors que $R$ est un sous-schéma ouvert de $T\times_{S}T$.\\

Notons qu'aucune des quatre propriétés énoncées pour $\pi_{0}$ n'est en général vérifiée pour $\pi^s$.\\

\subsection{Le morphisme $\pi_{0}(T/S) \to \pi^s(T/S)$}

La propriété universelle $(c)$ de $\pi_{0}$ montre qu'il existe un morphisme d'espaces algébriques $\theta: \pi_{0}(T/S) \to \pi^s(T/S)$, que l'on peut préciser:

\begin{prop} On suppose que $S$ est normal intègre et que le morphisme $T\to S$ est lisse. Soit $R$ l'ouvert de $T\times_{S}T$ qui est le graphe de la relation d'équivalence qui définit $\pi_{0}(T/S)${\rm (\ref{LMB})}. Alors l'adhérence schématique $\overline{R}$ de $R$ dans $T\times_{S}T$ est le graphe de la relation d'équivalence qui définit $\pi^s(T/S)$.

De plus, tout morphisme de $S$-espaces algébriques $\pi_{0}(T/S) \to F$, où $F$ est séparé, se factorise de fa\c con unique par $\theta$; en d'autres termes, $\theta$ fait de $\pi^s(T/S)$ l'enveloppe séparée de $\pi_{0}(T/S)$.
\end{prop}

Rappelons qu'un espace algébrique étale \emph{et séparé} est (représentable par) un schéma (\cite[A.2]{LMB00}), ou (\cite[6.16]{Knu71}).
\begin{proof} Que l'adhérence $\overline{R}$ soit le graphe ouvert et fermé d'une relation d'équivalence est établi dans (\ref{adhR}). Montrons que cette relation est minimale parmi les relations à graphe ouvert fermé dans $T\times_{S}T$, i.e. que c'est la relation d'équivalence $R'$ qui définit $\pi^s(T/S)$. L'existence du morphisme $\pi_{0}(T/S) \to \pi^s(T/S)$ entraîne l'inclusion $R \subset R'$; puisque le sous-schéma $R'$ est fermé, il contient $\overline{R}$, et la minimalité de $R'$ implique l'égalité cherchée.

La démonstration de la deuxième partie est analogue à celle de la fin de \ref{p.12}: le diagramme commutatif suivant résume la situation.
$$
\xymatrix{R \ar@{=}[r] & T\times_{\pi_{0}}T \ar[r]^{j} \ar[d] & T\times_{\pi^s}T \ar@{-->}[dl]^w \ar[d]^{\phi} \ar@{=}[r] & \overline{R}\\
& T\times_{F}T \ar[r]_{\psi} & T\times_{S} T&}
$$
Les morphismes $\phi$ et $\psi$ sont des immersions ferm\'ees puisque $\pi^s$ et $F$ sont s\'epar\'es sur $S$, et  $j$ est \scd\,  (\ref{im.sch}); on en tire  l'existence du morphisme $w$ rendant les triangles commutatifs. On termine la démonstration comme dans \ref{p.12}.
\end{proof}

\begin{para} Sous les hypothèses du lemme, il est vraisemblable que le morphisme $\theta: \pi_{0}(T/S) \to \pi^s(T/S)$ soit un isomorphisme local (local sur la source); d'après la proposition \ref{p.12}, c'est vrai si $\pi_{0}(T/S)$ est un schéma; mais dans l'exemple qui suit l'espace algébrique $\pi_{0}(T/S)$ n'est pas un schéma.\end{para}

\begin{ex}\label{ex} Cet exemple est signalé dans \cite[6.8.1]{LMB00}; c'est l'ouvert de lissité de la courbe de {\sc Mumford} citée p.210 de \cite{BLR90}. Il met en lumière quelques différences entre les foncteurs $\pi_{0}$ et $\pi^s$. 

On pose $S = \s(A)$, où $A =  {\bf R}[[t]]$,  et on prend pour $T$ le complémentaire de l'origine dans $\s(A[X, Y]/(X^{2} + Y^{2} -t))$.

Le morphisme $f : T\to S$ est lisse, sa fibre générique $T_{\eta} \to \eta$ est géométriquement irréductible. La fibre fermée $T_{s} \to s$ est isomorphe au morphisme 
$$
\s({\bf C}[X]_{X}) \to \s({\bf R}).
$$ 

On en tire les propriétés suivantes.
\begin{itemize}
\item (i) Pour toute factorisation $T \by{h} E \by{g} S$, où $h$ est surjectif et $g$ étale (séparé ou non), $g$ est un isomorphisme; en particulier $\pi^s(T/S) = S$;
\item (ii) le morphisme canonique $\pi^s(T_{s}/s)= \s({\bf C})  \to \pi^s(T/S)_{s}=\s({\bf R})$ n'est pas un isomorphisme (l'isomorphisme $\pi^s(T_{\xi}/\xi) \simeq \pi^s(T/S)_{\xi}$ de \ref{p11}, établi pour la fibre générique, n'est donc plus vrai pour la fibre fermée);
\item (iii) la fibre fermée du morphisme $h: T \to \pi^s(T/S)=S$ est irréductible mais pas géométriquement connexe (alors que la fibre fermée de $T\to \pi_{0}(T/S)$ est géométriquement connexe);
\item (iv) l'espace algébrique $\pi_{0}(T/S)$ \emph{n'est pas un schéma}.
\end{itemize}
\bb

(i) Comme la fibre générique $T_{\eta} \rrr \eta$ est géométriquement irréductible et que $h_{\eta}: T_{\eta}\to E_{\eta}$ est fidèlement plat, le morphisme
étale $E_{\eta} \to \eta$ est  géométriquement irréductible, donc  un isomorphisme; en particulier,  les points fermés de $E$ sont dans la fibre fermée $E_{s}$. Le morphisme $T_{s}\to E_{s}$ est fidèlement plat et  $T_{s}$ est connexe, l'espace discret $E_{s}$  a donc un seul point, noté $e$ ;  comme c'est l'unique point fermé de $E$ le morphisme canonique $\s(\oo_{E, e})\rrr E$ est un isomorphisme. Cela montre que le morphisme $g$ est affine fidèlement plat et birationnel; c'est un isomorphisme.

(ii) et (iii). Clair.

(iv) (Nous devons cet argument à {\sc M. Romagny}). Notons $K = \kappa(\eta)$ le corps des fractions de $A$, et posons $P = \pi_{0}(T/S)$. Cet espace algébrique est étale et commute aux changements de base ; lorsque la base est le spectre d'un corps $k$, cet espace est le spectre de la cl\^oture séparable de $k$ dans l'anneau des sections globales du schéma. Comme $T_{\eta} \to \eta$ est géométriquement irréductible, $K$ est algébriquement fermé dans $\Gamma(T_{\eta})$, donc $P_{\eta} \, \simeq \, \s(K)$. Comme la fibre fermée $T_{s} \to s$ est isomorphe à $\s({\bf C}[X]_{X}) \to \s({\bf R})$ la fermeture algébrique de ${\bf R}$ dans ${\bf C}[X]_{X}$ est égale à ${\bf C}$.  Cela montre que $P$ n'est pas un schéma. En effet, soit $e$ l'unique point de $P_{s}$ ; si $P$ était un schéma étale sur le trait hensélien $S = \s({\bf R}[[t]])$, le morphisme composé
 $$
 \s(\oo_{P, e}) \rrr P \rrr S
 $$
 serait fini \cite[18.5.11, $c')$]{EGAIV4}, donc libre ; ce qui est impossible puisque le rang générique est 1, et le rang spécial 2.
\end{ex}


 \appendix
 
 \section{Autres démonstations du théorème \ref{t1}}\label{sA}

Rappelons l'énoncé du théorème en question.

{\it Soit $f : T \rrr S$ un morphisme plat et de présentation finie de schémas. Soit $d : R \rrr T\times_{S}T$ une immersion ouverte et fermée, graphe d'une relation d'équivalence dans $T$. Alors le faisceau fppf quotient $T/R$ est représentable par un schéma $E$ quasi-compact étale et séparé sur $S$. Le morphisme $R \to T\times_{E}T$ est un isomorphisme, i.e. la relation d'équivalence est effective.}
 
Le schéma quotient de l'énoncé est construit comme un ouvert d'un schéma affine sur $S$. 

 Notons que si $T \to S$ est un morphisme quasi-compact étale et séparé, et si on prend pour relation le morphisme diagonal lui-même $T \to T\times_{S}T$, on obtient directement le résultat bien connu que $T$ est quasi-affine sur $S$. 
En admettant ce point, ainsi que l'effectivité des données de descente fpqc pour de tels schémas \cite[IX, 4.1, p.182]{SGA1}, {\sc L. Moret-Bailly} propose une troisième démonstration, plus courte, de ce théorème ; elle est donnée en \ref{3.5}.

Nous donnons ici une démonstration directe de ce  théorème: elle n'utilise que  les définitions et les propriétés les plus élémentaires des objets introduits, et pas les résultats profonds de {\sc M. Artin}.

\subsection{Notations et lemmes préliminaires}\

Par souci de référence, nous adoptons les conventions d'indices proposées par {\sc P. Gabriel} dans \cite[V, \S\S1 \`a 3]{SGA3}. 
En particulier, les morphismes de projection entre produits,  $p_{k} :T^n \rrr T^{n-1}$ sont index\'es de $0$ \`a $n-1$, l'indice $k$ d\'esignant la composante \emph{omise} ; ainsi, $p_{0}(x_{0}, x_{1}, x_{2}) = (x_{1}, x_{2})$.

Mais, contrairement à {\sc Gabriel}, et en  suivant un usage répandu, nous notons $R$ et $R'$ ce qu'il note $T_{1}$  et $T_{2}$.
\medskip

\n {\it Dans l'écriture des puissances d'un  $S$-schémas, la lettre $S$ en indice sera désormais omise.}\\

Une relation d'équivalence dans un $S$-schéma $T$ est la donnée d'un sous-schéma $d: R \to T^2$ tel que pour tout $S$-schéma $Z$ l'ensemble $R(Z) = \Hom_{S}(Z, R)$ soit le graphe d'une relation d'équivalence dans $T(Z)$; le schéma $R \subset T^2$ représente les $(x, y) \in T^2(Z)$ qui sont équivalents (notation $x\sim y$). On détaille ces données et hypothèses en  \ref{par1}, \ref{par2} et \ref{par3}.  
\medskip
 
\begin{quote}\begin{para}\phantomsection\label{par1}(Réflexivité) {\it Une relation d'équivalence sur $T$  comporte une  immersion $d : R \to T^2$, et un morphisme de $S$-schémas $s: T \to R$ qui factorisent le morphisme diagonal:}
 $$
 d \small{\circ} s \, = \, \Delta_{T/S} .
 $$
 \end{para}
 \end{quote}
 
\begin{quote}\begin{para}\phantomsection\label{par2} (Symétrie) {\it L'automorphisme $\sigma$ de  permutation  des facteurs de $T^2$ stabilise $R$.}
\end{para}\end{quote}
\bb
 
 Pour énoncer ce qui correspond à la propriété de \og transitivité\fg,
 on introduit les morphismes $d_{0} = p_{0}d$, et $d_{1}=p_{1}d : \xymatrix{R \ar@<0.5ex>[r]\ar@<-0.5ex>[r]& T}$,
ainsi que le schéma
$$
R' \, = \,  (R, d_{0})\times_{T}(R, d_{1}) 
$$
Utilisant le fait que le carré suivant est cartésien
$$
\xymatrix{T^3 \ar[r]^{p'_{0}} \ar[d]_{p'_{2}} &T^2 \ar[d]^{p_{1}}\\
T^2 \ar[r]_{p_{0}} &T}
$$
on peut écrire $ R' \, = \,  {p'_{0}}^{-1}(R) \cap {p'_{2}}^{-1}(R)$, et
on dispose donc d'une immersion  $d' : R' \, \rrr\, T^3$ qui permet d'identifier $R'$ à un sous-schéma de $T^3$;  en notant $d'_{i}$ les morphismes induits par les projections $p'_{i}$, on a alors  $d'_{0}(x, y, z) = (y, z)$, $d'_{1}(x, y, z) = (x, z)$ et $d'_{2}(x, y, z) = (x, y)$ ; ainsi $R'(Z)$ s'identifie à l'ensemble des triplets $(x, y, z)$ tels que $x\sim y$ et $y \sim z$; la \emph{transitivité} se traduit donc par l'inclusion\\

\begin{para}\phantomsection\label{par3} $ {p'_{0}}^{-1}(R) \cap {p'_{2}}^{-1}(R)  \subset {p'_{1}}^{-1}(R).$\\

Précisons que le symbole ${p'_{i}}^{-1}(R)$ désigne ici le sous-schéma $(R, d)\times_{T^2}(T^3, p'_{i})$ de  $T^3$, image réciproque de $R$ par $p'_{i}: T^3 \to T^2$.

\end{para}

 \bb
 
Le résultat suivant résume les propriétés générales qui seront utilisées..

\begin{lemme}\phantomsection\label{l1} \emph{ \cite[V,1, p.257]{SGA3}}  Soit $R$ une relation d'équivalence dans le $S$-schéma $T$. Alors dans le diagramme $$
  \xymatrix{R'   \ar@<0.5ex>[r]^{d'_{1}}  \ar@<-0.5ex>[r]_{d'_{0}} \ar[d]_{d'_{2}} & R \ar[d]^{d_{1}}  \ar[r]^{d_{0}} & T\\
  R  \ar@<0.5ex>[r]^{d_{1}}  \ar@<-0.5ex>[r]_{d_{0}}&T& }
  $$ la première ligne est exacte, et $R'$ s'identifie au produit fibré $(R, d_{0})\times_{T}(R, d_{0})$; les deux carrés de gauche de même indice sont cartésiens. \qed
  \end{lemme}

\begin{lemme}\phantomsection\label{l2} Soit  $f : T\rrr S$ un $S$-schéma, et
$$\xymatrix{R \ar@<0.5ex>[r]^{d_{1}} \ar@<-0.5ex>[r]_{d_{0}}& T}
$$
une relation d'équivalence dans $T$ dans la catégorie des $S$-schémas. Par le changement de base $T \rrr S$, on obtient la relation d'équivalence dans $T\times T$ (dans la catégorie des schémas sur $T$, via $p_{0}$)
$$
\xymatrix{R\times T \ar@<0.5ex>[r]^{d_{1}\times 1} \ar@<-0.5ex>[r]_{d_{0}\times 1} &T\times T}
$$
Considérons $R$ (resp. $R'$) comme schéma sur $T$ via $d_{0}$ (resp. via $d_{0}d_{0}' = d_{0}d_{1}'$), et notons $d'' : R' \rrr R\times T$ le morphisme de composantes $d'_{2}$  et $d_{0}d'_{0}$. Alors, dans le diagramme commutatif 
$$
\xymatrix{R'  \ar[d]_{d''} \ar@<0.5ex>[r]^{d'_{1}}  \ar@<-0.5ex>[r]_{d'_{0}} & R \ar[d]^{d}\\
 R\times T \ar@<0.5ex>[r]^{d_{1}\times 1} \ar@<-0.5ex>[r]_{d_{0}\times 1} &T\times T} 
 $$
 les deux carrés sont cartésiens.\qed
\end{lemme}
 
\subsection{Enveloppe affine}\

  Soit $f : T \rrr S$ un morphisme quasi-compact et quasi-séparé de schémas, de sorte que $f_{\star}(\oo_{T})$ est une $\oo_{S}$ algèbre quasi-cohérente. L'\emph{enveloppe affine} du $S$-schéma $T$ est le schéma affine sur $S$ 
$$
T^{{\rm aff}} \, = \, \s_{S}(f_{\star}(\oo_{T})),
$$
muni de son morphisme canonique 
$$
i_{T} : T \; \rrr \; T^{{\rm aff}}
$$
Voir (\cite{EGAI} 9.1.21, où $T^{{\rm aff}}$ est noté $T^{\rm 0}$). L'application $Z \mapsto  i_{T}^{-1}(Z) = Z\times_{T^{{\rm aff}}}T$ établit une bijection entre les ensembles des sous-schémas ouverts et fermés de $T^{{\rm aff}}$ et de $T$ (\ref{stein}), 
 $$
 i_{T}^{\star}: \of(T)\;  \rt \;  \of(T^{{\rm aff}}) .
$$
\medskip

Cette construction s'étend aux relations d'équivalence:

\n Soit  \; $d_{0}, d_{1} : \xymatrix{R \ar@<0.5ex>[r] \ar@<-0.5ex>[r]& T}$  une relation d'équivalence dans le $S$-schéma $T$ ; supposons que les morphismes canoniques   $f : T\rrr S$ et $g : R \rrr S$ soient quasi-compacts et quasi-séparés ; les $\oo_{S}$-algèbres $f_{\star}(\oo_{T})$ et $g_{\ast}(\oo_{R})$ sont donc quasi-cohérentes, ainsi que, par suite,  l'algèbre
$$
\mathcal{A} = {\rm Ker}( \xymatrix{f_{\star}(\oo_{T}) \ar@<0.5ex>[r] \ar@<-0.5ex>[r]& g_{\ast}(\oo_{R})}) .
$$
L'enveloppe affine de la relation d'équivalence (ou de son faisceau quotient) est par définition le $S$-schéma affine
$$
A = {\s}_{S}(\mathcal{A}) .
$$

\begin{lemme}\phantomsection\label{l3} Le morphisme composé $\alpha : T  \rightarrow T^{{\rm aff}} \rightarrow A$ est schématiquement dominant, i.e. l'application $\oo_{A} \rrr \alpha_{\star}(\oo_{T})$ est injective.
\end{lemme}

\begin{proof}  En effet, l'image directe de cette application par le morphisme $\beta : A \rrr S$ est l'application 
$$
\beta_{\star}(\oo_{A}) = \mathcal{A} \rrr \beta_{\star}\alpha_{\star}(\oo_{T}) = f_{\star}(\oo_{T}),
$$
 laquelle est injective par définition, et le morphisme $\beta$ est affine.
\end{proof}

\begin{lemme}\phantomsection\label{l4} La suite de morphismes $\xymatrix{R  \ar@<0.5ex>[r]^{d_{1}} \ar@<-0.5ex>[r]_{d_{0}} & T  \ar[r]^{\alpha} &A}$ induit, par image réciproque, une suite \emph{exacte} d'ensembles
$$
\of(A)\;  \to \; \of(T)\;  \rightrightarrows \; \of(R) .
$$
Autrement dit, si $W$ est un sous-schéma ouvert et fermé de $T$ tel que $d_{0}^{-1}(W) =d_{1}^{-1}(W)$, alors, il existe un unique sous-schéma ouvert et fermé $V$ dans $A$ tel que $\alpha^{-1}(V) = W$.
\end{lemme}
\begin{proof} Reprenons la suite exacte de $\oo_{S}$-algèbres quasi-cohérentes
$$
 \xymatrix{\mathcal{A} \ar[r] &f_{\star}(\oo_{T}) \ar@<0.5ex>[r] \ar@<-0.5ex>[r]& g_{\ast}(\oo_{R})} .
$$
Passant aux spectres, on obtient la suite de $S$-schémas affines
$$
\xymatrix{A & T^{{\rm aff}}  \ar[l] & R^{{\rm aff}}\ar@<0.5ex>[l] \ar@<-0.5ex>[l] }
$$
Cette suite n'est pas toujours exacte mais elle induit sur les ensembles d'ouverts fermés une suite exacte puisque les ouverts fermés de $\s(\mathcal{A})$ correspondent aux idempotents de $\Gamma(S, \mathcal{A})$. Pour conclure la démonstration, il suffit d'utiliser l'isomorphisme $ \of(T)\;  \rt \; \of(T^{{\rm aff}})$.
\end{proof}

\subsection{Démonstration du théorème \ref{t1}}\

On suppose dans la suite que $f : T \rrr S$ est un morphisme plat de présentation finie; il est donc, en particulier, quasi-compact et quasi-séparé (\cite[6.3.7]{EGAI}). On considère le graphe d'une relation d'équivalence dans $T$
$$
d : R \rrr T\times_{S} T ,
$$
o\`u $d$ est une immersion ouverte et fermée.

Comme plus haut, le schéma de base $S$ sera sous-entendu dans l'écriture des produits fibrés de $S$-schémas.
 \medskip

Notons $F$ le faisceau fppf quotient de $d_{0}, d_{1}$ et considérons le diagramme commutatif de faisceaux 
$$
\xymatrix{R \ar[d]_{i_{R}} \ar@<0.5ex>[r]^{d_{1}} \ar@<-0.5ex>[r]_{d_{0}} & T \ar[d]_{i_{T}} \ar[r]^p \ar[dr]^{\alpha}& F \ar@{-->}[d]^j\\
R\af  \ar@<0.5ex>[r] \ar@<-0.5ex>[r] & T\af \ar[r] & A}
$$ 
\bb

La ligne du haut est exacte par définition du quotient $F$, alors que celle du bas ne l'est en général pas ; mais on dispose cependant d'un  morphisme de faisceaux $j : F \rrr A$. On va montrer que sous les hypothèses du théorème, ce morphisme $j$  est représentable par une immersion ouverte ; cela entraînera que $F$ est (représentable par) un schéma, et que ce schéma est quasi-affine sur $S$.
\bb

\n Par le changement de base $T \rrr S$, on obtient la relation d'équivalence de $T$-schémas sur $T\times T$
$$
\xymatrix{R\times T \ar@<0.5ex>[r]^{d_{1}\times 1} \ar@<-0.5ex>[r]_{d_{0}\times 1} &T\times T}
$$
Son enveloppe affine (sur $T$) est $\alpha \times 1 : T\times T \rrr A\times T$, puisque $T$ est plat sur $S$ (\ref{pr1.1}).
\

Le diagramme suivant illustre les données relatives à la relation $R$, et à $\alpha$ ; le carré de droite est cartésien, ainsi que les carrés de gauche de même indice. 

Le morphisme composé $R \stackrel{d}{\to} T\times T \stackrel{\alpha \times 1}{\to} A \times T \to T$ est égal à $d_{0}$.
$$
\xymatrix{R'  \ar[d]_{d''} \ar@<0.5ex>[r]^{d'_{1}}  \ar@<-0.5ex>[r]_{d'_{0}} & R \ar[d]^{d} & \\
 R\times T \ar[d] \ar@<0.5ex>[r]^{d_{1}\times 1} \ar@<-0.5ex>[r]_{d_{0}\times 1} &T\times T \ar[d] \ar[r]_{\alpha\times 1}& A\times T \ar[d]\\
 R  \ar@<0.5ex>[r]^{d_{1}} \ar@<-0.5ex>[r]_{d_{0}} & T  \ar[r]_{\alpha} &A}  \leqno{(\star)}
 $$

 Considérons l'immersion ouverte et fermée $d : R \to T\times T$; ses images réciproques  par $d_{0}\times1$ et par $d_{1}\times 1$ sont égales à $d''$ puisque les deux carrés à gauche sont cartésiens. D'après le lemme \ref{l4}, il existe donc une immersion  ouverte et fermée $\delta : V \to A\times T$ dont l'image réciproque par $\alpha\times 1$ est égale à $d$ ; elle donne lieu au carré cartésien
 $$
 \xymatrix{R \ar[d]_{d} \ar[r]^{v} & V\ar[d]^{\delta}\\
 T\times T \ar[r]_{\alpha \times 1} & A\times T}
 $$
 où le morphisme $v$ est \scd, tout comme $\alpha \times 1$.\\
 
 Considérons le diagramme obtenu en composant les morphisme verticaux composables de ($\star$):
 $$
\xymatrix{R'  \ar[d]_{d'_{2}} \ar@<0.5ex>[r]^{d'_{1}}  \ar@<-0.5ex>[r]_{d'_{0}} & R \ar[d]^{d_{1}} \ar[r]^{v}&V \ar[d]^{\varphi} \\
 R  \ar@<0.5ex>[r]^{d_{1}} \ar@<-0.5ex>[r]_{d_{0}} & T  \ar[r]_{\alpha} &A} 
 $$
 où $\varphi$ est le morphisme composé $V \stackrel{\delta}{\to} A\times T \stackrel{{\rm pr}_{1}}{\to} A$; les trois carrés sont cartésiens. Puisque $\phi$ est plat et de présentation finie son image est un ouvert; ce morphisme se factorise donc en $V \by{\psi} U \by{\iota} A$, où $\psi$ est fidèlement plat de présentation finie et où $\iota$ est une immersion ouverte.
  Comme $d_{1}$ est surjectif, le morphisme $\alpha$ se factorise lui aussi par $\iota$, soit $\alpha = \iota u$; on obtient finalement le diagramme 
 $$
\xymatrix{R'  \ar[d]_{d'_{2}} \ar@<0.5ex>[r]^{d'_{1}}  \ar@<-0.5ex>[r]_{d'_{0}} & R \ar[d]^{d_{1}} \ar[r]^{v}&V \ar[d]^{\psi} \\
 R  \ar@<0.5ex>[r]^{d_{1}} \ar@<-0.5ex>[r]_{d_{0}} & T  \ar[r]_{u} &U}
 $$ 
 Il s'agit de montrer que $u : T \to U$ représente le quotient $T/R$. Or, $\psi$ est un morphisme fppf, et les trois carrés du diagramme sont cartésiens (pour le carré de droite, $\{\psi v ,  u d_{1}\}$, il faut se souvenir que $\{\iota \psi v, \iota u d_{1}\} = \{\phi v, \alpha d_{1}\}$ est cartésien, et que $\iota$ est un monomorphisme !). Il suffit donc  de montrer que $v : R \to V$ représente le quotient pour la ligne du haut. Or,  le lemme \ref{l1} montre que ce quotient est donné par $d_{0} : R \to T$ ; le morphisme $v$ se factorise donc  en $R \stackrel{d_{0}}{\to} T \stackrel{w}{\to}V$, et il reste à voir que $w$ est un isomorphisme.
 \\
 On a signalé plus haut que  $v$ est \scd \,; cela implique que $w$ l'est aussi. Par ailleurs, Le morphisme $w: T\to V$ est une section du morphisme composé $V \to A\times T \to T$, lequel  est séparé puisque $\delta : V \to A\times T$ est une immersion et que $A\times T \to T$ est affine ; $w$ est donc une immersion fermée. Mais une immersion fermée \scd e est un isomorphisme. $\Box$


\subsection{Troisième démonstration du théorème \ref{t1}, par {\sc L. Moret-Bailly}}\label{3.5}\

Pour un schéma $X$ sur un corps $k$, on définit le \emph{nombre géométrique de composantes connexes} de $X$, $n'(X)$ comme le nombre de composantes connexes de $X\otimes_{k}\Omega$ pour une (quelconque) extension algébriquement close $\Omega$ de $k$ \cite[4.5]{EGAIV2}.

\begin{lemme} \phantomsection\label{l5} Soit $ f : T\to S$ un morphisme surjectif de présentation finie, avec $S$ quasi-compact. Le nombre géométrique de composantes connexes des fibres  $T_{s} = f^{-1}(s)$, pour $s$ parcourant $S$, est une famille bornée, i.e.  $m(f) = {\rm max}(n'(T_{s}), s\in S)$ est fini.
\end{lemme}

\begin{proof} On peut supposer que  $S$ est affine, puis que $f$ provient, par un changement de base $S \to S_{0}$, d'un morphisme $f_{0} : T_{0} \to S_{0}$ o\`u $S_{0}$ est noethérien et $f_{0}$ de type fini ; on aura alors $m(f) \leq m(f_{0})$ ; on peut donc supposer que $S$ est affine et noethérien. En vertu de \cite[9.7.8]{EGAIV3}, pour tout $s \in S$, il existe un ouvert $U$ contenant $s$ tel que pour tout $s' \in U \cap \bar{s}$, on ait $n'(T_{s'}) = n'(T_{s})$.
Un raisonnement classique par récurrence noethérienne permet alors de construire une partition finie  $(S_{\alpha})$ de $S$ par des sous-schémas localement fermés de $S$, telle que les applications $s \mapsto n'(T_{s})$ soient constantes sur chaque $S_{\alpha}$ ; pour plus de détails, voir \cite[O$_{\rm I}$, 2.5.2]{EGAI}. 
\end{proof}

On reprend dans la suite les hypothèses du théorème \ref{t1}; en particulier, $f : T\to S$ est un morphisme plat de présentation finie. On raisonne par récurrence sur $m(f)$.

\begin{lemme}\phantomsection\label{l6} Si $m(f) = 1$, alors le quotient $T/R$ est isomorphe à $S$.
\end{lemme}

\begin{proof} Les fibres de $T\times_{S}T \to T$ sont géométriquement connexes, et rencontrent $R$ puisque la relation d'équivalence contient la diagonale ; mais $R$ est un ouvert fermé de $T\times_{S}T$, donc $R = T\times_{S}T$ ; comme $f$ est fidèlement plat quasi compact, la suite $\xymatrix{T\times_{S}T \ar@<0.5ex>[r]  \ar@<-0.5ex>[r] & T  \ar[r]&S}$ est exacte, i.e. on a $T/R = S$.
\end{proof}

\begin{para} Supposons que $f$ admette une section $g : S \to T$. Le saturé de $g$ est un ouvert fermé $U$ de $T$, puisqu'il est défini par le carré cartésien
$$
\xymatrix{U \ar[rr] \ar[d] && R \ar[d]^d\\
T \ar@{=}[r] & S\times_{S}T \ar[r]_{g\times 1} & T\times_{S}T}
$$
Notons $V$ l'ouvert fermé complémentaire, de sorte que $T = U \sqcup V$. Notons $R_{U}$  et $R_{V}$ les relations induites par $R$ sur $U$ et $V$ respectivement. On vérifie formellement que $R_{U} = U\times_{S}U$ et donc que le quotient  $U/R_{U}$ est isomorphe à $S$. D'autre part, chacune des fibres $T_{s}$ rencontre $U$ puisque $U$ contient la section $g$ ; on en déduit l'inégalité stricte $m(V\to S) < m(T\to S)$. L'hypothèse de récurrence entra\^ine que le quotient $V/R_{V}$ est un schéma étale et séparé sur $S$, et on peut conclure que $T/R = U/R_{U} \sqcup V/R_{V}$ est lui aussi un schéma étale et séparé sur $S$.
\end{para}

\begin{para} Cas général. On utilise \emph{l'universalité} des faisceaux $T/R$ sur ${\sf Sch}_{S}$, ce qui signifie ceci : pour un morphisme $S' \to S$, le faisceau sur ${\sf Sch}_{S'}$ déduit de $T/R$ par image réciproque est isomorphe au faisceau $S'\times_{S}T /S'\times_{S}R$ ; il est, en particulier, canoniquement muni d'une donnée de descente relative $ S'\to S$.

Appliquons au problème initial sur $S$, le changement de base $T \to S$ ; au-dessus de $T$ le morphisme $f$ acquiert une section ; d'après le point précédent, le faisceau devient donc représentable par un schéma étale et séparé sur $T$, donc quasi-affine sur $T$, et ce schéma est muni d'une donnée de descente relative à $T\to S$ . Mais d'après (\cite[IX, 4.1, p.182]{SGA1}), une donnée de descente sur un tel schéma est effective, autrement dit, le schéma en question provient de $S$.\qed
\end{para}

 
\bigskip

\noindent IMJ-PRG, Case 247\\
4 place Jussieu, 
75252 Paris Cedex 05,
France\\
\texttt{daniel.ferrand@imj-prg.fr}

\end{document}